\documentclass{amsart}
\usepackage{amsmath}
\usepackage{xcolor}
  \usepackage{graphics}
  \usepackage{amsfonts}
  \usepackage{epsfig}
  \usepackage{float}
  \usepackage{graphicx}
 \usepackage{epstopdf}
 \usepackage{amssymb}
 \usepackage{mathrsfs}
 \usepackage{amsthm}
 \usepackage[colorlinks=true]{hyperref}
 \usepackage[nameinlink, capitalise]{cleveref}
 \usepackage[numbers]{natbib}
\usepackage{color}
\usepackage{dsfont}

\hypersetup{urlcolor=blue, citecolor=red}
\usepackage{hyperref}
\hypersetup{%
  colorlinks = true,
  linkcolor  = black
}

\usepackage{comment}

\numberwithin{equation}{section}

\newtheorem{thm}{Theorem}[section]
\newtheorem{lem}[thm]{Lemma}
\newtheorem{prop}[thm]{Proposition}
\newtheorem{defi}[thm]{Definition}
\newtheorem{cor}[thm]{Corollary}
\newtheorem{rmk}[thm]{Remark}
\newtheorem{exm}[thm]{Example}

\newtheorem*{thm*}{Theorem} 
\newtheorem{theorem}{Theorem}

\DeclareMathOperator{\id}{id}

\newcommand{\HH}{\mathcal{H}}     
 
\newcommand{\HHb}{\mathcal{H}_{\bar{X}}}   
\newcommand{\YY}{\mathcal{Y}}     
\newcommand{\II}{\mathcal{I}}     
\newcommand{\BB}{\mathcal{B}}     
\newcommand{\hh}{\mathfrak{h}}    
\newcommand{\ttt}{\mathfrak{t}}   
\newcommand{\bb}{\mathfrak{b}}    
\newcommand{\PP}{\mathcal{P}}     
\newcommand{\RR}{\mathbb{R}}      
\newcommand{\ZZ}{\mathbb{Z}}      
\newcommand{\NN}{\mathbb{N}}      

\newtheorem{corollary}[theorem]{Corollary}
\newtheorem{lemma}[theorem]{Lemma}

\theoremstyle{definition}

\newtheorem*{conjecture*}{Conjecture}

\def\bar{\overline}

\usepackage[pagewise, mathlines]{lineno} \linenumbers 
\usepackage{etoolbox}          

\newcommand*\linenomathpatch[1]{%
  \cspreto{#1}{\linenomath}%
  \cspreto{#1*}{\linenomath}%
  \csappto{end#1}{\endlinenomath}%
  \csappto{end#1*}{\endlinenomath}%
}
\newcommand*\linenomathpatchAMS[1]{%
  \cspreto{#1}{\linenomathAMS}%
  \cspreto{#1*}{\linenomathAMS}%
  \csappto{end#1}{\endlinenomath}%
  \csappto{end#1*}{\endlinenomath}%
}
\expandafter\ifx\linenomath\linenomathWithnumbers 
 \let\linenomathAMS\linenomathWithnumbers
\patchcmd\linenomathAMS{\advance\postdisplaypenalty\linenopenalty}{}{}{}
\else
  \let\linenomathAMS\linenomathNonumbers
\fi

\linenomathpatch{equation}
\linenomathpatchAMS{gather}
\linenomathpatchAMS{multline}
\linenomathpatchAMS{align}
\linenomathpatchAMS{alignat}
\linenomathpatchAMS{flalign}
\makeatletter
\patchcmd{\mmeasure@}{\measuring@true}{
  \measuring@true
  \ifnum-\linenopenaltypar>\interdisplaylinepenalty
    \advance\interdisplaylinepenalty-\linenopenalty
  \fi
  }{}{}
\makeatother 

\usepackage{time}
\nolinenumbers

\title{Cocycle superrigidity for median spaces of finite rank}

\author{Biao Ma}
\address{Biao Ma \newline
School of Mathematical Sciences, Key Laboratory of Intelligent Computing and Applications(Ministry of Education), Tongji University, China\\
}
\email{24024@tongji.edu.cn}

\author{Lamine Messaci  }
\address{Lamine Messaci \newline
Laboratoire de recherche AGM, CY Cergy-Paris Université, France\\
}
\email{mmessaci@cyu.fr}

\begin{document} 
\begin{abstract}
We systematically investigate cocycle superrigidity in the setting of finite rank median spaces for product groups and Kazhdan groups. By employing a dynamical approach to superrigidity, we establish, for a median space $X$ of finite rank, the superrigidity of $Isom(X)$-valued cocycles for a product of locally compact second countable groups. In the case of actions by irreducible lattices in such product groups, this approach yields a novel proof of the superrigidity of homomorphisms.
\end{abstract}

\maketitle 

\tableofcontents

\section{Introduction}

Superrigidity phenomena have long served as powerful manifestations of algebraic and dynamical rigidity in group actions. Since Margulis’s foundational work on homomorphism superrigidity for lattices in higher-rank semisimple Lie groups\cite{margulis1991discrete}, and Zimmer’s extension to cocycles into algebraic targets \cite{Zimmer}, the theory has significantly evolved to encompass geometric contexts, especially through actions on spaces with nonpositive curvature. Over the past two decades, cocycle superrigidity theorems have been established for isometry groups of a variety of spaces—including Hilbert spaces, CAT(0) spaces, and Gromov-hyperbolic spaces—using tools from bounded cohomology, harmonic map theory, and boundary theory (see, for instance, \cite{Monod}, \cite{Monod-Shalom}, \cite{Shalom-superrigidity}, \cite{BaderFurman}).\par

This paper investigates cocycle superrigidity for isometry groups of a broad and flexible class of spaces: \textbf{finite rank median spaces}. Median spaces are metric spaces where each triple of points admits a unique median point. Median spaces generalize both real trees and CAT(0) cube complexes and have become central in geometric group theory due to their combinatorial and convexity properties. In particular, they are related to Kazhdan's Property (T) and a-T-menability, and provide a unified framework for analyzing hierarchically hyperbolic groups, coarse median spaces, and asymptotic cones of mapping class groups (see \cite{ChaDH_median}, \cite{ChaD_hyperbolic}, \cite{Fior_tits}, \cite{Fior_median_property}, \cite{Fior_superrigidity}, \cite{Fior_auto}, \cite{Bowd_conv}, \cite{Bowd_median-algebras}, \cite{Genevois}, \cite{Messaci} for different results and properties of the median geometry). Meanwhile, median geometry has also found its applications in understanding Cremona groups \cite{LonjouUrech}, birational groups \cite{CdC}, and computer science \cite{BandeltChepoi}. \par

We prove cocycle superrigidity theorems in this setting by developing a dynamical and boundary-theoretic approach. Unlike previous related works (\cite{Shalom-superrigidity}, \cite{Monod}, \cite{Monod-Shalom}, \cite{Gelander-Karlsson-Margulis} and \cite{Fior_superrigidity}), our strategy builds on ideas introduced by Bader--Furman \cite{BaderFurman} and is based on the nonpositive curvature features of group actions on median spaces of finite rank. In particular, it does not rely on bounded cohomology, nor on the existence of an induced $L^1$-structure.

\subsection{Statement of the results}
\subsubsection{Product groups}
Let $G=G_1\times G_2 \times \cdots \times G_n (n \geq 2)$ be a product of second countable, locally compact, non compact groups, and let $\Omega$ be a standard probability space on which each $G_i$ acts ergodically by measure-preserving action. The reader is referred to Section \ref{section_strongboundaries_cocycles} for cocycles and its non elementarity. A median space is called \textbf{irreducible} if it is not the $\ell^1$-product of two median spaces that are not one point space.
 \begin{theorem}[Theorem \ref{theoerem_superrigidity_cocycle}]\label{thm_intro_cocycles}
 Let $X$ be a complete separable irreducible median space of finite rank and $\alpha:G\times \Omega \rightarrow Isom(X)$ a non-elementary cocycle. 
 Then, there exists an $Isom(X)$-invariant median subspace $Y\subseteq X$ such that the cocycle $\alpha$ is cohomologous to a group homomorphism $h:G\rightarrow Isom(Y)$ that factors through the projection onto one factor $G_i$.
 \end{theorem}\par 
The $Isom(X)$-invariant median subspace $Y$ arising in Theorem \ref{thm_intro_cocycles}, that we refer to as the \textbf{\textit{median core}}, has the property that the restriction of any isometry of $X$ on $Y$ is uniquely determined by its restriction on its boundary. In general, a homomorphism from $G$ to $Isom(Y)$ cannot be lifted to $Isom(X)$ (see Example 4.7 in \cite{Fior_superrigidity}). Meanwhile, the median core $Y$ is in general not convex, hence its existence does not contradict the Roller minimality assumption. In \cite{ChaIF} and \cite{Fior_superrigidity}, the median core appeared implicitly in the proof of the superrigidity theorem. In this paper, we characterize it as the median hull of a particular part of the Roller boundary, called the \textbf{regular boundary}. \par
 
 We note that the superrigidity of cocycles is new even when the target group is the automorphism group of a finite dimensional CAT(0) cube complex:\par 

 \begin{corollary}
    Let $X$ be an irreducible separable CAT(0) cube complex of finite dimension and $\alpha:G\times \Omega \rightarrow Isom(X)$ a non-elementary cocycle. 
 Then, there exists an $Isom(X)$-invariant median subgraph $Y\subseteq X$ such that the cocycle $\alpha$ is cohomologous to a group homomorphism $h:G\rightarrow Isom(Y)$ that factors through the projection onto one factor $G_i$. 
 \end{corollary}

By taking $\Omega = G/\Gamma$ and $\alpha$ to be the returning cocycle for an irreducible lattice in $\Gamma< G$, the standard argument gives the superrigidity of homomorphism of lattices (see \cite{Furstenberg-Bourbaki}). This gives rise to a new argument that does not rely on the non vanishing of a bounded cohomology or induced cohomology class. It also recovers Fioravanti's superrigidity result for actions on median spaces of finite rank \cite{Fior_superrigidity} and extends it to the class of non square-integrable lattices. Recall that a lattice $\Gamma$ in a product $G_1\times G_2 \times \cdots \times G_n$ of locally compact second countable groups is called \textbf{irreducible} if for any $1 \leq i \leq n$, the projection of $\Gamma$ into $\prod_{j \neq i}G_j$ is dense. As a corollary of Theorem \ref{thm_intro_cocycles}, we have 
 \begin{theorem}[Theorem \ref{thm_superrigidity_principal}]\label{thm_intro_homo}
 Let $\Gamma\leq G_1\times G_2 \times \cdots \times G_n (n \geq 2)$ be an irreducible lattice in a product of locally compact second countable groups. If $\Gamma$ acts Roller minimally and Roller non-elementarily on a complete separable irreducible median space of finite rank, there exists an $Isom(X)$-invariant median subspace $Y$ on which the $\Gamma$-action extends to $G$ and factors through one of the $G_i$'s.
 \end{theorem}

We remark that Theorem \ref{thm_intro_homo} remains true for the weaker assumption of irreduciblity of the lattice being the projection into each factor $G_i$ is dense. \par
 As a corollary of Theorem \ref{theoerem_superrigidity_cocycle}, one obtains the cocycle version of \cite[Theorem C]{Fior_superrigidity}. 
\begin{theorem}\label{thm_elementary_cocycyle_product_group}
    Let $X$ be a complete irreducible separable median space of finite rank and let $G:=G_1\times \cdots \times G_n$ where each $G_i$ is a locally compact second countable group such that $G_i/G_i^0$ is either amenable or has Kazhdan's Property (T). Then any cocycle $\alpha:G\times \Omega\rightarrow Isom(X)$ is elementary.
\end{theorem}

 The main step in the proof of Theorem \ref{thm_intro_cocycles} is the proof of the existence and uniqueness of a boundary map from $\Omega \times B$, where $B$ is the $G$-boundary (see Definition \ref{def_G_boundary}), to the regular boundary $\partial_r X$ of the median space $X$, that is the part of the Roller boundary that encodes the tree-like part of the median space. We state its simplified version involved in the proof of Theorem \ref{thm_intro_homo}.
 
 \begin{theorem}[Theorem \ref{thm_image_in_regular_boundary}]
 Let $G$ be a locally compact second countable group acting Roller non elementarily and Roller minimally on an irreducible finite rank median space. Then there exists a unique $G$-equivariant measurable map $\phi:B\rightarrow \partial_r X$.
 \end{theorem}
For results on boundary maps in the median setting, see \cite{ChaIF}, \cite{Fernos}, \cite{Sageev-Nevo} and \cite{Bader-Taller}.
 The regular boundary $\partial_r X$ can be defined as the limit set of a particular set of contracting isometries. These are isometries whose axes of translation skewer through two strongly separated halfspaces, where two halfspaces are said to be strongly separated if there is no halfspace that is transverse to both. The existence of such elements is ensured under the mild assumptions that the median space is irreducible and the isometry group acts Roller non elementarily and Roller minimally (see Definitions \ref{def_Roller_minimal}).

 \par 

To construct the boundary map, we adapt the approach of \cite{ChaIF} to the median setting. Their method relies on halfspaces to associate a “center of mass” to each measure. A key step in their proof involves establishing the measurability of certain maps, which depends on the fact that the space of halfspaces of the cube complex is countable. This is obviously not necessarily true for general median spaces. However, we prove the existence of a countable dense subset of halfspaces that is invariant under the full group of isometries.
\begin{lemma}[Theorem \ref{thm_countable_medians_subspace}]\label{thm_fundamental_family_introduction}
Let $X$ be a complete separable irreducible median space of finite rank. If $Isom(X)$ acts Roller non elementarily and Roller minimally, there exists an $Isom(X)$ invariant countable set of halfspaces that separate any pair of points of the median core.

\end{lemma}

The theorem above is very practical for dealing with the measurability question regarding maps and subsets in the Roller compactification.\par

\subsubsection{Kazhdan groups.}

Product groups include products of simple Lie groups, but they exclude simple Lie groups of higher rank. These Lie groups have Kazhdan's Property (T). In \cite{ChaDH_median}, the authors gave a characterization of Property (T) for lcsc groups, namely, every continuous isometric action has bounded orbits. Hence, some rigidity properties may be expected for cocycles of Kazhdan groups with values in the isometry group of median spaces. In \cite{AdamsSpatzier}, the authors gave a complete study of cocycles of Kazhdan groups taking values in the isometry group of a tree, where they proved in particular that any such cocycle is cohomologous to a cocycle taking values in the stabilizer of a point. Applying the argument of Adams and Spatzier to the median setting, we extend our study of cocycle rigidity to groups satisfying Kazhdan’s Property (T).
\begin{theorem}[Theorem \ref{thm:cocycle-propertyT}]
    Let $X$ be a complete median space, and let $G$ be a locally compact second countable group with Kazhdan's Property (T) that acts ergodically by measure preserving action on $\Omega$. Then for any cocycle $\alpha:G\times\Omega\rightarrow Isom(X)$, there exists an $\alpha$-invariant field of bounded subset, i.e. an $\alpha$-invariant section $f:\Omega\rightarrow Bdd(X)$, such that $\{(\omega,x)\in \Omega\times X\ |\ x\in f(\omega)\}$ is measurable. 
\end{theorem}

A direct consequence of this theorem is
\begin{corollary}\label{intro:cocycle_cube_propertyT}
   Let $G$ be a lcsc group having Kazhdan's Property (T) and $(\Omega,\nu)$ an ergodic $G$-space. Let $X$ be a CAT(0) cube complex and $H \leq Isom(X)$. Then any measurable cocycle $\alpha: G \times \Omega \to H$ is cohomologous to a cocycle in the isotropy group in $H$ of a bounded sub-cube complex of $X$.
\end{corollary}

In the context of finite rank \textbf{connected} median space, we can prove the following fixed point property for Kazhdan group in the cocycle setting.
\begin{theorem}[Theorem \ref{thm_equivariant_cocycle_Kazhdan}]\label{thm_equivariant_cocycle_Kazhdan_introduction}
     Let $X$ be a complete separable connected median space of finite rank and let $G$ be a lcsc group with Kazhdan's Property (T) that acts ergodically by measure preserving action on $(\Omega,\mu)$. Then for any cocycle $\alpha:G\times\Omega\rightarrow Isom(X)$ there exists an $\alpha$-equivariant section $f:\Omega\rightarrow X$. \par 
     In particular, $\alpha$ is elementary.
\end{theorem}
If the cocycle $\alpha$ is $L^1$-integrable, that is, for any $g\in G$ the integral $$\int_\Omega d(x_0,\alpha(g,\omega)(x_0)) \ d\mu(\omega)$$ is finite for some (any) fixed $x_0\in X$, then the group $G$ acts isometrically on the space of $L^1$-integrable sections class from $\Omega$ into $X$. Hence in this case, Theorem \ref{thm_equivariant_cocycle_Kazhdan_introduction} follows directly from the median characterization of the Kazhdan property (T) (\cite{ChaDH_median}). \par
In the general case, the cocycle $\alpha$ induces a $G$-action on the space of equivalence classes of measurable sections from $\Omega$ into $X$, which in general carries only the structure of a median algebra, rather than a median metric.



\color{black}
\subsection{Median spaces of infinite rank.}
Many interesting examples of median spaces are not of finite rank such as the canonical median spaces associated to real and complex hyperbolic spaces constructed in \cite{ChaD_hyperbolic}. Other interesting classes of infinite rank median spaces are given by infinite-dimensional CAT(0) cube complexes (median graph of infinite rank). In \cite{LonjouUrech}, the authors constructed  various infinite-dimensional CAT(0) cube complexes on which Cremona groups act by isometries (see also \cite{Genevois-Lonjou-Urech}). In \cite{Osajda}, the author proved that various classes of free Burnside groups do not have Kazhdan's Property (T) by constructing infinite-dimensional CAT(0) cube complexes on which they act properly. Theorem \ref{thm_elementary_cocycyle_product_group} no longer holds when the finite rank assumption is dropped as one may consider the action of a product of $\mathrm{SO}(n,1)$ on the product of the canonical median spaces associated to $\mathbb{H}^n$ (see \cite{ChaD_hyperbolic}). An important distinction is that, for infinite rank median spaces, the isometry group is not necessarily totally disconnected, whereas it is the case for "generic case" for median spaces of finite rank.
\par 
\color{black}
Regarding the superrigidity of cocycles for product groups, the dynamics approach taking here indicates that the same result should conjecturally hold for infinite rank median spaces, at least for the subcategory of median spaces admitting a nice topological separation by halfspaces (such as these in \cite{ChaD_hyperbolic} and infinite rank median graph). The proof we provided in this paper relies on the finite rank assumption, particularly for the existence of strongly separated halfspaces, which gives rise to an abundant family of contracting elements in $Isom(X)$. In finite rank median space, strongly separated halfspace exists as long as the space is irreducible. This is no longer true in the infinite rank case, where strongly separated halfspaces do not necessarily exist even under irreducibility condition, consider, for instance, the canonical median space associated to the $n$-dimensional hyperbolic space \cite{ChaD_hyperbolic}.\par 

One interest in extending the superrigidity result to the infinite rank case would be to establish the property FW for the remaining cases of higher rank irreducible lattice. The conjecture was stated for irreducible lattices in connected semisimple Lie group of higher rank with no compact factors (see \cite[Conjecture 1.9]{Cornulier}), and it is established in the case where one of the factors has Kazhdan's Property (T) (see \cite[Theorem 1.12]{Cornulier}). For other examples of lattices verifying property FW (see \cite[section 1.2]{Cornulier-2}).

    

\subsection{Outline of the proof}
As mentioned above, Theorem \ref{thm_intro_homo} is a corollary of Theorem \ref{thm_intro_cocycles}. However, the paper is organized to first prove Theorem \ref{thm_intro_homo} and then Theorem \ref{thm_intro_cocycles} by pointing out minor changes. It has the advantage that no background on cocycles is needed to understand the main ideas. We now present an outline of their proofs. For simplicity, in this paper we only deal with products of two groups, and the general version can be easily completed by induction.

\subsubsection{Superrigidity of homomorphisms (Section \ref{section_superrigidity_of_homomorphisms}).}

Let $\Gamma\leq G_1\times G_2$ be an irreducible lattice in a product of locally compact second countable groups and let $X$ be a complete, irreducible, separable median space of finite rank. Assume that $\Gamma$ acts Roller non elementarily and Roller minimally on $X$. For a fixed $G$-boundary $B$, the first step is to construct a $G$-equivariant map $\phi:B\rightarrow \bar X$ (Theorem \ref{thm_contruction_boundary_map}). To do so, one first starts from the $G$-equivariant map $\Phi:B\rightarrow \PP(\bar X)$, where $\PP(\bar X)$ is the space of probability measures over the Roller compactification $\bar X$. Such $G$-equivariant map is ensured by the amenability of the action $G\curvearrowright B$. For each probability measure $\mu\in \PP(\bar X)$ one associates a center of mass by considering the intersection of all closed halfspaces having measure bigger than $\frac{1}{2}$ (see Section \ref{section_barycenter_map}). Using the ergodicity of the diagonal $G$-action on $B\times B$, one deduces that for almost all $b\in B$, the center of mass of the corresponding measure is a singleton (Proposition \ref{prop_balanced_halfspaces_zero_measure}). Thus, the map $\Phi$ descends to a measurable map $\phi: B\rightarrow \bar X$.\par 
The second step consists of showing that the map $\phi$ has an image essentially in the regular boundary (Theorem \ref{thm_image_in_regular_boundary}). To do so, one uses again the double ergodicty of the action $G\curvearrowright B\times B$ to deduce that for almost all pair $b_1,b_2$ their images under $\phi$ are separated by an arbitrarily large set of strongly separated hyperplanes. The regular boundary has the property that for any triple of pairwise distinct points, their median point lies in $X$.\par 
 Having these data in hand and borrowing an argument from Bader-Furman \cite{BaderFurman}, one shows that any $G$-equivariant map from $B\times B$ into $\partial_r X$ depends on one variable and factors through $\phi$ (Theorem \ref{thm_unicity_of_boundary_map_product}). Restricting to the situation where the group is an irreducible lattice in a product of locally compact second countable groups, one has $B=B_1\times B_2$ where each $B_i$ is a $G_i$-boundary. In this case, the product $B\times B$ contains a non ``trivial" automorphism that commute with the action (the generalized Weyl group introduced in \cite{Bader-Furman-Weyl_group} is bigger than $\ZZ/2\ZZ$). Combining this with the proven fact that any map from $B\times B$ to $\partial_r X$ depends on one variable, this imposes the boundary map $\phi$ to depend only on one factor $B_i$ (Theorem \ref{thm_factor_boundary_map}). Hence, one obtains a $G_i$-action on $\phi(B)\subseteq \partial_r X$ extending the $\Gamma$-action. The last step is to show that the $G_i$-action extends continuously to the median closure of $\phi(b)$ inside $X$, using the median ternary operation and the density of the projection of $\Gamma$ in $G_i$ (Theorem \ref{thm_superrigidity_principal}).

 \subsubsection{Superrigidity of cocycles (Section \ref{section_superrigidity_of_cocycles}).}

 Let $G=G_1\times G_2$ be a product of second countable locally compact groups, and let $\Omega$ be a standard probability space on which each $G_i$ acts ergodically with measure-preserving transformations. Let $X$ be a complete separable irreducible median space of finite rank and let $$\alpha:G\times \Omega \rightarrow Isom(X) $$ be a non elementary cocycle (see Section \ref{section_strongboundaries_cocycles}). The cocycle defines a natural $G$-action on the space of classes of measurable sections from $\Omega$ to the median core $Y$. This space inherits a natural structure of a median algebra, where the median operation is defined pointwise.\par  
 The goal, inspired by \cite{bader2006superrigidity}, is to use boundary theory along with the commutativity between the factors $G_1$ and $G_2$, to prove the existence of a $G$-invariant median subalgebra $\mathcal{Y}$ which, when endowed with the $L^1$-metric, is a median space isometric to $Y$. Moreover, this median subalgebra $\YY$ is such that one of the factors, say $G_2$ acts trivially on it. \par 
 This median subspace consists of classes of $\alpha_{G_2}=\alpha|_{G_2}$-equivariant sections from $\Omega$ to $Y$. To prove the existence of such $\alpha_{G_2}$-equivariant sections, we start with an $\alpha$-equivariant section $\Phi:\Omega\times B\rightarrow \PP(\bar X)$, ensured by the amenability of the action $G\curvearrowright \Omega\times B$. Adapting the same arguments used in the superrigidity of homomorphisms to the case of cocycle, we show that the map $\Phi$ descends to an $\alpha$-equivariant map $\phi:\Omega\times B\rightarrow \partial_r X$ that factors through the projection into $\Omega\times B_1$ (Proposition \ref{prop_boundary_map_cocycle}, Proposition \ref{prop_boundary_map_factors}). Almost all fixed $b_1\in B_1$ gives rise to an $\alpha_{G_2}$-equivariant section from $\Omega$ into $\partial_r X$. After composing such sections with the ternary operation, one obtains infinitely many $\alpha_{G_2}$-equivariant sections from $\Omega$ to $Y$, denoted by $\YY$. The final step is then to show that such a 
 subspace is isometric to $Y$ (Theorem \ref{theoerem_superrigidity_cocycle}), using the ergodicity of the $G_i$-actions on $\Omega$.


\subsection{Organization of the paper}
In Section \ref{Section_median_geometry}, we recall definitions and properties of median geometry and prove some results in the separable case that will be used in the sections that follow. In particular, Subsection \ref{subsection_fundamental_family} is devoted to the proof of Lemma \ref{thm_fundamental_family_introduction}.

In Section \ref{section_strongboundaries_cocycles}, we briefly recall generalities on boundary theory and cocycles. We also adapt the notion of non-elementarity of cocycles to the median setting.

In Section \ref{section_barycenter_map}, we construct the barycenter map in the median setting, which assigns to each measure its center of mass, and prove the necessary propositions for its use in the construction of boundary maps (Subsections \ref{subsection_construction_boundary_map} and \ref{subsection_boundary_map_cocycle}).

Section \ref{section_superrigidity_of_homomorphisms} is devoted to the proof of a superrigidity result for actions by irreducible lattices $\Gamma \leq G_1 \times G_2$. Subsection \ref{subsection_construction_boundary_map} focuses on the existence and uniqueness of the boundary map for Roller non-elementary and Roller minimal actions of countable groups.

Section \ref{section_superrigidity_of_cocycles} is devoted to the proof of the superrigidity of cocycles for products of locally compact second countable groups $G_1 \times G_2$. In analogy with Subsection \ref{subsection_construction_boundary_map}, Subsection \ref{subsection_boundary_map_cocycle} focuses on the proof of existence and uniqueness of the boundary map for non-elementary cocycles.\par
In Section \ref{section_Kazhdan_group}, we deal with cocycles for Kazhdan groups.
\subsection*{Acknowledgments.}  
The authors thank Indira Chatterji for her valuable comments on an earlier version of the paper, which helped improve both its substance and clarity. The first-named author is supported by the National Natural Science Foundation of China (Grant No. 12401086).
\section{Geometry of median spaces and compactifications}\label{Section_median_geometry}

\subsection{Median spaces and median algebras}
\begin{defi}[Median spaces]
    A metric space $(X,d)$ is a (metric) median space if, for any $x,y,z\in X$, there exists a unique point $m = m(x,y,z) \in X$, called the median point of $x,y,z\in X$, such that 
    \[
    [x,y]\cap[x,z]\cap [y,z]=\{m\}
    \]
    where, for any two points $a,b$ in $X$, the interval $[a,b]$ between $a$ and $b$ is defined to be
    \[
    [a,b]:=\{c\in X\ | \ d(a,b)=d(a,c)+d(c,b)\}.
    \]
\end{defi}
\par
We provides a collection of examples of median spaces (cf. \cite{ChaDH_median} for more examples).
\begin{exm}
    \begin{itemize}
        \item A real tree is a median space. Moreover, it is the only case where the intervals coincide with the geodesics.
        \item The 1-skeleton of a CAT(0) cube complex endowed with a combinatorial metric is a median graph. By \cite{Chepoi}, any median simplicial graph arises as the 1- skeleton of a CAT(0) cube complex.
        \item The space of $L^1$-integrable functions over a measurable space $(A,\BB,\mu)$ is a median space \cite{ChaDH_median}.
        \item By \cite{Bowd_conv}, any asymptotic cone of a coarse median group admits a natural median structure. Hence, the asymptotic cones of hyperbolic groups, mapping class groups are median spaces.
    \end{itemize}
\end{exm}

As for CAT(0) cube complexes, a key feature of median geometry is the concept of halfspaces. Recall that a subset $A \subseteq X$ in a median space $X$ is called \textbf{convex} if $[a,b] \subseteq A$, whenever $a, b \in A$. 
\begin{defi}
    Let $X$ be a median space. A subset $\hh\in X$ is a \textbf{\textit{halfspace}} if both $\hh$ and $\hh^c$ are convex. 
\end{defi}

We denote by $\HH(X)$ the subset of halfspaces of $X$. Note that the subset $\HH(X)$ is not empty. In fact, the following Separation Theorem implies that there are enough halfspaces in $X$, in the sense that for any two disjoint convex subsets in a median space,  there is a halfspace to separate them.

\begin{thm}[Separation Theorem \cite{Roller}]
    Let $X$ be a median space, then for any two disjoint convex subsets $C_1,C_2\subseteq X$, there exists a halfspace $\hh\in\HH(X)$ such that $C_1\subseteq \hh$ and $C_2\subseteq \hh^c$.
\end{thm}

\par
In general, median spaces are not uniquely geodesic spaces. However, convex subsets are rigid enough to impose some nice property on the geometry of the space. In particular, they verify the Helly property.
\begin{thm}[Helly's property \cite{Roller}]
    Let $X$ be a median space and $C_1,...,C_n$ be pairwise-intersecting convex subsets in $X$. Then the intersection $\displaystyle{\bigcap_{i=1}^n}C_i$ is not empty.
\end{thm}
The separation theorem and Helly's property constitute the cornerstone of the median geometry. They are the key ingredient in the construction of a natural compactification (bordification) of the median spaces through halfspaces (see, for instance, \cite{Roller} and \cite{Fior_median_property}). For a subset $A \subset X$, we use both $\mathring A$ and $int(A)$ to denote the interior of A. 

\begin{defi}[Half-space interval]\label{defi_nomenclature_halfspaces}
    Let $X$ be a median space. 
    \begin{itemize}
        \item  For any subsets $A,B\subseteq X$, we set 
    \begin{eqnarray*}
    \HH(A,B)&:=&\{\hh\in\HH(X)\ |\ B\subseteq \hh \ \& \ A\subseteq \hh^c\}  \\
    \mathring \HH(A,B)&:=& \{\hh\in\HH(X)\ |\ B\subseteq  int(\hh) \ \& \ A\subseteq int(\hh^c)\} 
    \end{eqnarray*}
    For singletons we simply write $\HH(x,y)=\HH(\{x\},\{y\})$ and $\mathring\HH(x,y)=\mathring \HH(\{x\},\{y\})$. If $\hh\in\HH(A,B)$ we say that $\hh$ \textbf{\textit{separates}} $B$ from $A$.
         \item   A halfspace $\hh\subseteq X$ is said to be \textbf{\textit{thick}} if $int(\hh) \neq \emptyset$ and $int(\hh^c) \neq \emptyset$.
         \item  Two disjoint halfspaces $\hh_1$ and $\hh_2$ are \textbf{\textit{strongly separated}} if there are no halfspaces transverse to them.
     \end{itemize}
\end{defi}

The set of halfspaces is canonically endowed with a structure of measured space $(\HH(X),\BB,\nu)$ where $\BB$ is the $\sigma$-algebra generated by $\{\HH(x,y)\}_{x,y\in X}$ and the measure $\nu$ verifies a Crofton like formula $\nu(\HH(x,y))=d(x,y)$ (cf. \cite{ChaDH_median} and \cite{Fior_median_property}).

\par
Any median space $X$ comes with a ternary operation $m_X:X^3\rightarrow X$ that associates to each triple of points their median point. When two variables of $m_X$ are fixed, it corresponds exactly to the nearest point projection onto intervals as stated in the following proposition.
\begin{prop}[\cite{ChaDH_median}] 
    Let $X$ be a complete median space, and $C\subseteq X$ a closed convex subset.
    \begin{itemize}
        \item The nearest point projection onto $C$, denoted by $\pi_C$ is a well-defined $1$-Lipschitz map from $X$ to $C$, and satisfies $\pi_C(x)\in[x,c]$ for any $x\in X$ and $c\in C$.
        \item If $C=[a,b]$, the nearest point projection $\pi_C$ is exactly the map $$m(a,b,*): x \mapsto m(a,b,x).$$ Moreover, intervals in median spaces are closed and convex.
    \end{itemize} 
\end{prop}
Note that the interval between $a,b\in X$ can also be defined by the mean of the median operation as the subset of fixed points under the map $m(a,b,*)$. \par Later in the paper, we will consider a compactification of the median space called the Roller compactification. The latter does not have a canonical median metric, but does come with a very useful canonical structure of median algebra that is defined as follows:


\begin{defi}[Median algebras]
    A median algebra $(M,m)$ is a set $M$ endowed with a ternary operation $m: M \times M \times M \to M $ satisfying:
    \begin{eqnarray*}
        m(x,x,y)  &=& x;\\
        m(x,y,z)  &=& m(y,x,z) = m(x,z,y);\\
  m(m(x,y,z),t,w) &=& m(x,m(y,t,w),m(z,t,w)).
    \end{eqnarray*}
\end{defi}
Note that any metric median space has a natural structure of a median algebra given by the ternary operation (see \cite{Sholander} for the equivalence of the axioms). We remark that the separation theorem and Helly's property come from the structure of median algebra. \par 
The rank of a median algebra reflects the ``dimension" of the median algebra. Recall that a halfspace $\hh$ is \textbf{\textit{transverse}} to a convex subset $C$ if both $\hh \cap C$ and $\hh^{c} \cap C$ are nonempty. We say that two halfspaces $\hh_1,\hh_2\subseteq M$ are \textbf{transverse} if the following intersections are nonempty:
\[
\hh_1\cap\hh_2,\ \hh_1\cap\hh_2^*,\ \hh_1^*\cap\hh_2,\ \hh_1^*\cap\hh_2^*.
\]
A family of pairwise transverse halfspaces gives rise to an embedded $0$-skeleton of an $n$-cube through a median morphim, i.e. a morphism that preserves the median operation. Hence, the transversality of halfspaces gives a natural way to characterize dimension in median algebras.
\begin{defi}[Rank]
Let $n \in \mathbb{N}$. We say that a median algebra $M$ is of \textbf{\textit{rank}} $n$ if $n$ is the maximal number of pairwise transverse halfspaces. If such $n \in \mathbb{N}$ exists, we call $M$ a \textbf{\textit{median algebra of finite rank}}.   
\end{defi}
The reader is referred to \cite{Bowd_conv} for other equivalent definitions of the rank of median algebras. The rank of a median space is defined as the rank of the associated median algebra. Median spaces of finite rank include, for instance, real trees and the 0-skeleton of finite-dimensional CAT(0) cube complexes endowed with the combinatorial metric. 
\par We now collect some facts about median algebras that will be used in the sequel. For median algebras, even if one does not necessarily have a notion of distance, one has an analogue of the nearest-point projection.
\begin{defi}
    Let $M$ be a median algebra. A \textbf{\textit{retract}} $C\subseteq M$ is a convex subset endowed with a retraction $\pi_C: M\rightarrow C$ such that for any $x\in M$ and any $c\in C$,
    \[
    \pi_C(x)\in[x,c].
    \]
\end{defi}

\begin{prop}[Lemma 2.13 \cite{ChaDH_median} and Lemma 2.6 \cite{Fior_median_property}]
    If $X$ is a median space or a compact topological median algebra, then a convex subset is a retract if and only if it is closed.
\end{prop}
The \textbf{\textit{convex hull}} of a subset $A\subseteq X$, denoted by $Conv(A)$, is the smallest convex subset containing $A$. The convex hull between two convex subsets $C_1,C_2\subseteq X$, denoted by $Conv(C_1,C_2)$, is the convex hull of their union.
\begin{prop}[Proposition 2.12 \cite{Roller} ]\label{prop_convex_hull_retract}
    If $M$ is a median algebra and $C_1,C_2\subseteq M$ are two retract subsets. Then the convex hull $Conv(C_1,C_2)$ is also a retract of $M$.
\end{prop}
As a consequence of the previous two propositions, the convex hull between two closed convex subsets is also closed.

\begin{lem}\label{lem_projection_intersect_gate}
If $X$ is a median algebra and let $C\subseteq X$ be a retract subset. Then for any convex subset $C'\subseteq X$ such that $C\cap C'\neq \emptyset$, we have $\pi_C(C')\subseteq C\cap C'$.\par 
In particular, we have $\pi_C(C')=C\cap C'$.
\end{lem}
\begin{proof}
    Let $c'\in C' $ and $c\in C\cap C'$. Since $C'$ is convex, we have $\pi_C(c')\in [c,c']\subseteq C'$ which implies that $\pi_C(c')\in C\cap C'$.
\end{proof}

The following proposition is needed to ensure the measurability of some subsets and maps that we will encounter in the next subsections.
\begin{prop}\label{prop_minimal_halfspace_countable_orig}
    Let $X$ be a complete separable median space of finite rank and let $C\subseteq X$ be a proper convex subset. Then the subset of minimal closed convex subsets containing $C$ is a non empty countable subset.
\end{prop}
Let us prove the above proposition in a larger setting.
\begin{prop}\label{prop_minimal_halfspace_countable}
    Let $X$ be a complete separable median space and $\HH'\subseteq \HH(X)$ be a familly of open halfspaces such that for any $\hh_1,\hh_2\in\HH'$, the halfspaces $\hh_1$ and $\hh_2$ are either disjoint or transverse. If there is $n\in \NN$ such that $\HH'$ contains no more than $n$ pairwise transverse halfspaces, then $\HH'$ is countable.
\end{prop}
\begin{proof}Since $X$ is separable, fix a dense countable subset $Y\subseteq X$ and proceed by induction on the number of maximal pairwise transverse halfspaces in $\HH'$. If all halfspace of $\HH'$ are disjoint, i.e., n = 1, then $\HH'$ is countable as we have a natural surjection from Y to $\HH'$.\par 
Let us assume that the proposition is true for some $n\in\NN$ and prove that it is true for $n+1$. We assume that $\HH'$ contains no $n+1$ halfspaces which are pairwise transverse. Assuming Zorn's lemma, consider a maximal subset $\HH''\subseteq \HH'$ consisting of pairwise disjoint halfspaces. Note that the subset $\HH''$ is countable by the same argument as in the first paragraph of the proof. By the maximality of $\HH''$, any $\hh\in \HH'\backslash\HH''$, is transverse to some $\hh'\in\HH''$. Hence we have 
\[
\HH'=\bigcup_{\hh\in \HH''}\HH_\hh^t,
\]
where $\HH_\hh^t$ is the subset of halfspaces in $\HH'$ that are transverse to $\hh$. Hence, it is enough to show that each $\HH_\hh^t$ is countable. The familly $\HH_\hh^t$ contains no $n$ halfspaces that are pairwise transverse. By the induction hypothesis, we conclude that each $\HH_\hh^t$ is countable, which completes the proof.

\end{proof}

\subsection{Roller compactification}\label{subsection_Roller}
There is a natural compactification for median spaces called the Roller compactification introduced by Fioravanti in \cite{Fior_median_property}. Let $X$ be a complete median space of finite rank. We have the following injective embedding of median algebra
\begin{eqnarray*}
i:X&\rightarrow& \II(X):=\prod_{a,b\in X}[a,b].\\
  x&\mapsto& (m(x,a,b))_{a,b\in X}
\end{eqnarray*}
Intervals in finite rank median spaces are compact, hence the median algebra $\II$ endowed with the product topology is a compact median algebra. The \textbf{\textit{Roller compactification}} $\bar{X}$ of $X$ is defined to be the closure of $i(X)$ in $\II(X)$. The boundary $\partial X = \bar{X} - i(X)$ of $i(X)$ in $\II(X)$ is called the \textbf{Roller boundary} of $X$.  
\begin{thm}[\cite{Fior_median_property}]
    The space $\bar{X}$ is a compact topological median algebra.
\end{thm}
The median operation on $\bar{X}$ is denoted by $\tilde{m}$. This compactification comes naturally with a monomorphisms of poc sets $i_\HH:\HH(X)\rightarrow\HH(\bar{X})$ which associates to each halfspace $\hh\in\HH(X)$ the following halfspace in the Roller compactification:
\[
\tilde{\hh}:=i_\HH(\hh)=\{x\in\bar{X}\ | \ \forall\ a\in \hh, \forall \ b\in \hh^c, \ \tilde{m}(x,a,b)\in\hh\}
\]
Remark that the intersection of $\tilde{\hh}$ with an interval $[\tilde{a},\tilde{b}]$ where $a,b\in X$ is exactly the image of $i(\hh\cap[a,b])$ as we have $[\tilde{a},\tilde{b}]=i([a,b])$ (see \cite[Corollary 4.4 (1)]{Fior_median_property}). 
\begin{prop}\label{prop_description_halfspace_arising_from_X}
    For any halfspace $\hh\in \HH(X)$ and $a,b\in X$ such that $a\in \hh$ and $b\in \hh^c$, we have $\tilde{\hh}=\tilde{\pi}_{[a,b]}^{-1}(\hh\cap[a,b])$, where $\tilde{\pi}_{[a,b]}$ is the retraction on $\bar{X}$ to $[a,b]$.\par
    In particular, any halfspace of $\bar X$ that is transverse to $i(X)$ is of the form $\tilde{\hh}$ for some $\hh\in X$.
\end{prop} 

\begin{proof}
This is a consequence of the fact that the inverse image of a halfspace under an median morphism is also a halfspace. 
\end{proof}

The following lemma states that any pair of distinct points in the Roller compactification are separated by the lift of a halfspace of $X$.
\begin{lem}\label{lem_separation_tilde_h}
    For any $\tilde{x},\tilde{y}\in\bar{X}$, there exists a halfspace $\hh\in \HH(X)$ such that $\tilde{\hh}\in\HH_{\bar{X}}(\tilde{x},\tilde{y})$.
\end{lem}
\begin{proof}
 We write $\tilde{x}=(x_I)_{I\in\II} $ and $\tilde{y}=(x_I)_{I\in\II} $. The points $\tilde{x},\tilde{y}\in\bar{X}$ being distinct, there exists an interval $I=[a,b]\subseteq X$ such that $x_I\neq y_I$. By \cite[Corollary 4.4 (2)]{Fior_median_property}, we have $i(x_I)=\tilde{\pi}_I(\tilde{x})$ and $i(y_I)=\tilde{\pi}_I(\tilde{y})$ (we abuse the notation by identifying $I$ with its embedding $i(I)$). Thus, $\tilde{\pi}_I(\tilde{x})$ and $\tilde{\pi}_I(\tilde{y})$ are distinct. Hence, the lift of any halfspace $\hh\in\HH(\tilde{\pi}_I(\tilde{x}),\tilde{\pi}_I(\tilde{y}))$ under $\tilde{\pi}_I^{-1}$ gives rise to a halfspace $\tilde{\hh}$ that separates $\tilde{y}$ from $\tilde{x}$.
\end{proof}
The Roller compactification $\bar{X}$ is isomorphic to the Busemann compactification (see \cite[Proposition 4.21]{Fior_median_property}). The space $\bar{X}$ is stratified into components where each component is a finite rank median space, and there is a unique component of maximal rank corresponding to the space $X$. Two points $\tilde{x},\tilde{y}\in \bar{X}$ are in the same component if the set of halfspaces that arises from $X$ and separates them is of finite $\nu-$measure. The stratification is preserved under the action of the isometry group. \par
\begin{lem}[Proposition 4.29 \cite{Fior_median_property}]\label{lem_halfspace_in_component}
    For any $x,y\in \bar{X}$ lying in the same component of $\bar{X}$, we have $\mathring \HH_{\bar{X}}(x,y) =\HHb(x,y)\cap i(\HH(X))$.
\end{lem}
\begin{prop}\label{prop_topology_Roller_compactification}
    The topology induced on $X$ from the Roller compactification is the same as the topology generated by open halfspaces.
\end{prop}
Before proving the above proposition, we need the following lemma.
\begin{lem}\label{lem_lineal_median_algebra_proper_embedding_Roller}
   Let $X$ be a complete median space of rank $n$ that contains no facing triple (i.e. a triple of halfspaces that are pairwise disjoint). Then the topology generated by open halfspaces coincides with the one generated by the metric.\par 
  In particular, if $X$ is an interval, then both topologies coincide.
  \end{lem}
\begin{proof}
    Let us fix a point $x\in X$ with $r>0$ and show that there exist open halfspaces $\hh_1,..,\hh_k\in \HH(X)$ such that they intersections contains $x$ and are contained in the ball $B(x,r)$. Let us denote by $\HH_x^t:=\{\hh\in\HH(X)\ | \ \hh\ \text{minimal and closed with} \ d(x,\hh)\leq t\}$. By Zorn's lemma, the subset $\HH_x^t$ is not empty. As the median space is of finite rank and contains no facing triple, $\HH_x^t$ contains finitely many minimal elements. By taking $t=\frac{r}{n}$, we deduce from \cite[Lemma 3.7]{Messaci}, that $\displaystyle{\bigcap_{\hh\in\HH_x^t}\hh^c}\subseteq B(x,r)$, which completes the proof.
\end{proof}
\begin{proof}[Proof of Proposition \ref{prop_topology_Roller_compactification}]
    The result comes from Lemma \ref{lem_lineal_median_algebra_proper_embedding_Roller} and the fact that the inverse image of an open halfspace under the projection into closed convex subsets is also an open halfspace.
\end{proof}
In Section \ref{section_superrigidity_of_homomorphisms}, we will need the following proposition.
\begin{prop}\label{prop_Borel_structure_Roller}
    Let $X$ be a complete separable median space of rank $n$. Then the topology induced from the Roller compactification induces the same Borel structure as the one induced from the median metric.
\end{prop}
\begin{proof}
    The convex hull of any ball of radius $r$ is contained in a ball of radius $n.r$ (see \cite[Proposition 3.6]{Messaci}). Thus, the space being separable, any open ball is the countable union of the closure of convex hull of balls. Hence, it is enough to show that any closed convex subset is obtained as a countable intersection of closed halfspaces. 
 Let $C\subseteq X$ be a closed convex subset. We set $\HH_C^m$ to be the subset of minimal closed halfspaces containing $C$. The space $X$ being separable, the family $\HH_C^m$ is countable by Proposition \ref{prop_minimal_halfspace_countable}. By the Separation Theorem, $C$ is obtained as the intersection of all elements $\HH_C^m$, which completes the proof. 
\end{proof}
We recall the definition of the regular boundary of an irreducible median space, which is the part of the boundary of $\bar{X}$ that encodes the tree-like behavior of the median space.
\begin{defi}
    Let $X$ be a complete irreducible median space of finite rank. A point $x\in \partial X$ is said to be \textbf{\textit{regular}} if there exists a sequence of halfspaces $(\hh_i)_{i\in\NN}$ such that for any $i\in\NN$ we have 
    \begin{itemize}
    \item $x\in \tilde{\hh}_i$,
    \item $\hh_{i+1}\subseteq \hh_i$,
    \item $\hh_i^c$ and $\hh_{i+1}$ are strongly separated (see Definition \ref{defi_nomenclature_halfspaces}), with $d(\hh_i^c,\hh_{i+1})>r$ for some $r>0$ that does not depends on the index $i$.
    \end{itemize}
    We call the set of regular points the \textbf{\textit{regular boundary}} and denote it by $\partial_r X$.
\end{defi}

Under mild regularities assumptions on the action of the isometry group of the median space, the median space contains uncountably many regular points. The regularity assumption ensures the existence of many contracting elements in the isometry group, and the limit set of such contracting isometries is constituted of regular points. The regularity of an action is essentially characterized by the following definitions.
\begin{defi}\label{def_Roller_minimal}
    Let $X$ be a complete median space of finite rank that is not a single point. We say that an isometric action $\Gamma\curvearrowright X$ is
    \begin{itemize}
        \item \textbf{\textit{Roller minimal}} if there is no proper $\Gamma$-invariant closed convex subset in the Roller compactification $\bar{X}$ of $X$.
        \item \textbf{\textit{Roller non elementary}} if there is no finite orbit in the Roller compactification $\bar{X}$ of $X$.
    \end{itemize}
\end{defi}

\begin{rmk}
Let $X$ be a complete irreducible median space of finite rank. If the action $Isom(X)\curvearrowright X$ is Roller minimal and Roller non elementary, then $\partial_r X$ is not empty (see Remark \ref{rmk_construction_regular_point}), and in fact, it is not difficult to see that it is uncountable.
\end{rmk}

One particularity of the regular boundary is that the median point of any triple of distinct points in it lies inside the median space, as shown in the following lemma. 
\begin{lem}\label{lem_interval_between_strongly_separated_ultrafilters}
Let $X$ be a complete finite rank median space and $\tilde{x},\tilde{y}\in\partial_r X$ two distinct points in the regular boundary. Then for any $\tilde{z}\in \bar{X}\backslash\{\tilde{x},\tilde{y}\}$, we have $m(\tilde{z},\tilde{x},\tilde{y})\in X$.
\end{lem}
\begin{proof}
    Let $(\hh_i)_{i\in\NN}$, $(\ttt_i)_{i\in\NN}$ be descending sequences of open halfspaces of $X$ such that $\hh_{i+1}$ and $\hh_{i}^c$ (respectively $\ttt_{i+1}$ and $\ttt_{i}^c$) are strongly separated, with $\displaystyle{\tilde{x}=\bigcap_{i\in\NN}\tilde{\hh}_i}$ and $\displaystyle{\tilde{y}=\bigcap_{i\in\NN}\tilde{\ttt}_i}$. Fix $\tilde{z}\in \bar{X}\backslash \{\tilde{x},\tilde{y}\}$ and assume without loss of generality that $z\in\hh_0^c\cap\ttt_0^c$. Consider then sequences of points $(\tilde{x}_i)_{i\in\NN_{\geq 1}}$ and $(\tilde{y}_i)_{i\in\NN_{\geq 1}}$ such that $x_i\in\hh_i\backslash\hh_{i+1}$ and $y_i\in\ttt_i\backslash\ttt_{i+1}$. Remark that $\pi_{\hh_0^c}(x_i)=\pi_{\hh_0^c}(x_j)$ and $\pi_{\ttt_0^c}(y_i)=\pi_{\ttt_0^c}(y_j)$, for any $i,j\in\NN_{\geq 1}$. Hence, we get
    \[
    m(\tilde{z},\tilde{x}_i,\tilde{y}_i)=m(\tilde{z},\tilde{x}_j,\tilde{y}_j)=m(\tilde{z},\tilde{x}_1,\tilde{y}_1)\in X.
    \]
    On the other hand, by the continuity of the median operation, we conclude that 
    \[
   m(\tilde{z},\tilde{x},\tilde{y}) = \lim_{i\rightarrow +\infty}  m(\tilde{z},\tilde{x}_i,\tilde{y}_i) = m(\tilde{z},\tilde{x}_1,\tilde{y}_1) \in X.
    \]
\end{proof}

The idea of generating rank 1 elements from elements that ``skewer" through two strongly separated halfspaces goes back to Caprace and Sageev \cite{SagC}, where they proved rank rigidity theorem for CAT(0) cube complexes. In \cite{Fior_tits}, Fioravanti extended the machinery to the case of median spaces.
\begin{prop}[Flipping Lemma (Corollary 5.4 \cite{Fior_tits})]\label{lem_flipping_lemma}
    Let $X$ be a complete irreducible median space of finite rank such that $Isom(X)$ acts Roller non elementarily and Roller minimally. Then for any thick halfspace $\hh$, there exists $g\in Isom(X)$ such that $g.\hh\subset \hh^c$. We say that $g$ \textbf{\textit{flips}} the halfspace $\hh$.
 \end{prop}
Note that here we slightly change the definition of flipping to allow the case $g.\hh=\hh^c$, as in the irreducible case one can show that, under the Roller minimality condition, any halfspace is flippable in the original sense (i.e. there exists $g\in Isom(X)$ such that $d(g.\hh,\hh)>0$).
 \begin{rmk} \label{rmk_construction_regular_point}

Let $X$ be an irreducible complete median space of finite rank such that $Isom(X)$ acts Roller non elementarily and Roller minimally. \par 
If $\hh_1$ and $\hh_2$ are two strongly separated halfspaces and $g_1,g_2\in Isom(X)$ are two elements that flip $\hh_1^c$ and $\hh_2^c$, respectively. Then the isometry $g_1\circ g_2$ is a rank 1 element which its ``translation axe" skewers through $\hh_1^c$ and $g.\hh_1$. Its contracting and repulsive points are respectively given by 
 \begin{eqnarray*}
     \tilde{x}&=&\bigcap_{i\in\NN}g^i.\tilde{\hh}_1, \\
     \tilde{y}&=&\bigcap_{i\in\NN}g^{-i}.\tilde{\hh}^c_1,
 \end{eqnarray*}
and are regular points.

 \end{rmk}

The following is a useful characterization of the regular points in $\bar X$.
\begin{lem}\label{lem_characterization_regular_points}
    Let $\tilde{x}\in \partial X$, $x_0\in X$, $r>0$ and assume that for any $n\in\NN$, there exists $\hh_1,..,\hh_n$ such that $x_0\in\hh_1^c$, $\tilde{x}\in\tilde{\hh}_n$ where $\hh_i^c$ and $\hh_{i+1}$ are strongly separated with $d(\hh_i^c,\hh_{i+1})>r$ for any $i\in\{1,2,\cdots,n\}$. \par 
    Then, $\tilde{x}$ is a regular point.
\end{lem}
\begin{proof}
    Let $\mathfrak{H}$ be the set of sequences of halfspaces satisfying the conditions stated in the lemma. We endow $\mathfrak{H}$ with the order given by the inclusion. Note that $\mathfrak{H}$ is an inductive set. By the definition of the regular point, it suffices to show that there exists a sequence in $\mathfrak{H}$ that is infinite. Hence, to prove the lemma it is enough to prove that maximal elements in $\mathfrak{H}$ are infinite, which is equivalent to show that any finite element in $\mathfrak{H}$ can be extended. Let us consider a sequence $(\hh_1,..,\hh_n)$ and note that $d(x_0,\hh_n)\geq (n-1)r$. According to the hypothesis of the lemma, there exists a sequence $(\hh_1',...,\hh_{N}')\in\mathfrak{H}$ where $N \in \NN$ is such that $(N-2)r>d(x_0,\hh_n)$, that is, $d(x_0,\hh_{N-1}')>d(x_0,\hh_n)$. As both $\tilde{\hh}_n$ and $\tilde{\hh}_{N-1}'$ contain $\tilde{x}$, $\hh_n$ and $\hh_{N-1}'$ are either transverse or $ \hh_{N-1}'\subseteq \hh_n$ . If $\hh_n$ and $\hh_{N-1}'$ are transverse, then we have $\hh_{N}'\subseteq \hh_n$ as $\hh'^c_{N-1}$ and $\hh'_{N}$ are strongly separated. In either case, we have $\hh_{N}'\subseteq \hh_n$. Hence, $(\hh_1,..,\hh_n,\hh_{N}')$ extends the sequence $(\hh_1,..,\hh_n)$ in $\mathfrak{H}$ which completes the proof.
\end{proof}

\begin{lem}\label{lem_separation_Roller_compactification}
    For any pair of distinct points $\tilde{x},\tilde{y}\in \bar{X}$, there exists a pair of points $x,y\in X$ such that $\HH_{\bar{X}}(x,y)\subseteq \mathring\HH_{\bar{X}}(\tilde{x},\tilde{y})\subseteq i_\HH(\HH(X))$. 
\end{lem}
\begin{proof}
    Note that for any $x,y\in X$, we have $\HH_{\bar{X}}(x,y)=i(\HH(x,y))$  (Proposition \ref{prop_description_halfspace_arising_from_X}). The points $\tilde{x}$ and $\tilde{y}$ being distinct, there exists an interval $I=[a,b]$ where $a,b\in X$, such that $x':=\pi_I(\tilde{x}) \in X $ and $y':=\pi_I(\tilde{y}) \in X$ are distinct. 
    Remark that we have the following inclusion $i(\HH(x',y'))\subseteq \HH_{\bar{X}}(\tilde{x},\tilde{y})$ (cf. Lemma \ref{lem_separation_tilde_h} ). \par
    Let us consider a closed convex neighborhood $C_1,C_2\subseteq [x',y']$ of $x'$ and $y'$ respectively such that $C_1\cap C_2 =\emptyset$ and set $x=\pi_{C_1}(y') \in X$ and $y=\pi_{C_2}(x') \in X$. Hence, we have $\HH(x,y)\subseteq \mathring\HH(x',y')$ and conclude that $i(\HH(x,y))\subseteq \mathring \HH_{\bar X}(\tilde{x},\tilde{y}) $.\par 
   The inclusion $\mathring\HH_{\bar{X}}(\tilde{x},\tilde{y})\subseteq i_\HH(\HH(X))$ comes from the fact that any open subset of $\bar X$ intersects $X$, hence, any halfspace in $\mathring \HH(\tilde{x},\tilde{y})$ is transverse to $X$, so it is induced from a halfspace in $\HH(X)$ (Proposition \ref{prop_description_halfspace_arising_from_X}).
\end{proof}

\subsection{Median core}\label{subsection_median_core}

In \cite{ChaIF} and \cite{Fior_superrigidity}, the superrigidity result for actions of lattices extends only on a median subspace that is invariant under the action. In general, one should not expect the extension to hold for the whole median space as described in  \cite[Example 4.7]{Fior_superrigidity}. The construction of the median subspace $Y$ is implicit in both \cite{ChaIF} and \cite{Fior_superrigidity}. In this subsection, we give an explicit construction of such median subspace that corresponds to the median closure of the regular boundary.

\begin{defi}[\textbf{\textit{Median core}}]\label{def_median_core}
    Let $X$ be a complete irreducible separable median space of finite rank. 
    Consider the following $Isom(X)$-invariant subset
  \[
    Y_0:=\{m_{\bar X}(\tilde{x}_1,\tilde{x}_2,\tilde{x}_3) \ |\ \tilde{x}_1,\tilde{x}_2,\tilde{x}_3\in  \partial_r X, \ x_i\neq x_j\}\subseteq X 
  \]
    and we define recursively $Y_{i+1}:= m(Y_i,Y_i,Y_i)$. We set $Y':=\displaystyle{\bigcup_{i\in\NN}Y_i}$ and remark that it is a $Isom(X)$-invariant median subspace.\par
   We call the closure of $Y'$ in $X$ the $\textbf{\textit{median core}}$ and denote the closure by $Y$.
\end{defi}

\begin{exm}
    Let $T$ be a simplicial tree. Fix an interval of any median space of finite rank and replace all edges of $T$ by such an interval to obtain a new median space $X$. In this case, the median core $Y$ of $X$ coincides with the vertices of $T$. One may obtain (depending on the interval) infinitely many isometries of $X$ that fix the vertices of $T$ and act as involutions along the intervals. Such isometries induce the identity on the Roller boundary of $X$. 
 \end{exm}
 
\begin{rmk}\label{remark_median_core}
  \begin{itemize}
      \item Any isometry of $X$ descends to an isometry of $Y$, and is uniquely determined by its restriction to $\partial_r X$.
      \item If $X$ is a complete irreducible median space of finite ranl, then its median core is also a complete irreducible median space of finite rank.
      \item  If $\Gamma$ acts Roller non elementarily and Roller minimally on $X$, then its induced action on the median core $Y$ is also Roller non elementary and Roller minimal. 
      \item Using the induction on an argument that is similar to the one used in the proof of \cite[Theorem B]{Messaci}, one can show that the stabilizer of any point $x\in Y'$ is open in $Isom(X)$ where $Y'$ is the median subspace arising in Definition \ref{def_median_core}. Here, $Isom(X)$ is endowed with the pointwise convergence topology.
  \end{itemize}
\end{rmk}

\subsection{\texorpdfstring{Countable $Isom(X)$-invariant median subspace and fundamental family}{Countable Isom(X)-invariant median subspace and fundamental family}} \label{subsection_fundamental_family}

The aim of this section is to construct a countable $Isom(X)$-invariant median subspace and a countable dense family of halfspaces. These sets are useful in showing measurability in the sequel. The density of the set of halfspaces is in the following sense.
\begin{defi}\label{def_fundamental_family}
    Let $X$ be a median space of finite rank. A \textbf{\textit{fundamental family of halfspaces}} $\HH'\subseteq\HH(X)$ is a countable family of halfspaces such that
    \begin{itemize}
     \item[(1)] For any $\tilde{x},\tilde{y}\in \bar{X}$, there exists $\hh\in\HH'$ such that $\tilde{\hh}\in\HH_{\bar{X}}(\tilde{x},\tilde{y})$.
        \item[(2)] For any thick halfspace $\hh\in\HH(X)$, there exists $\hh'\in \HH'$ such that $\hh\subseteq \hh'$.
    \end{itemize}
\end{defi}
We proceed as follows: from a dense countable subset of $X$, we construct a dense countable family of halfspaces. Using the latter family, we prove that the intersection of the median hull of the regular boundary with $X$ is countable.\par

\begin{prop}\label{prop_existence_fundamental_family}
   Any complete separable median space $X$ of finite rank contains a fundamental family of halfspaces.
\end{prop}
\begin{proof}
Fix a dense countable subset $Y\subseteq X$. For each $x\in Y$ and $n\in\NN_{\geq 1}$ we set the following subset of halfspaces
\[
\HH_{x,n}:=\{\hh\in\HH(X)\ | \ Conv(B(x,\frac{1}{n}))\subseteq \hh \ \text{with} \ \hh \ \text{closed and minimal}\}
\]
The subsets $\HH_{x,n}$ are countable by the separability of $X$ (Proposition \ref{prop_minimal_halfspace_countable}). We set 
\[
\HH':=\bigcup_{\substack{x\in Y \\ n\in\NN^*}}\HH_{x,n}
\]
 and claim that the subset $\HH'$ is a fundamental family.\par
 Let us first show that any interval that is distinct from a point is separated by a halfspace of $\HH'$. \par 
 We fix $\tilde{x} \neq \tilde{y}\in \bar{X}$. Then there exists an interval $I=[a,b]$ with $a, b \in X$, such that $x:=\pi_I(\tilde{x})$ is distinct from $y:=\pi_I(\tilde{y})$. Let us consider a point $y'\in Y$ which is close enough to $y$ (if $y$ is isolated then $y\in Y$, in which case we choose $y'$ to be $y$) and consider $n\in\NN_{\geq 1}$ such that $C:=\bar{Conv(B(y',\frac{1}{n}))}$ contains $y$ and does not contain $x$. Then there exists a halfspace $\hh\in\HH_{y',n}$ which separates $C$ from $x$, that is $\hh\in\HH(x,y)$. This implies $\tilde{\hh}\in\HHb(\tilde{x},\tilde{y})$ and finishes the first claim of the proposition.\par 
 Then we prove that $\HH'$ defined above satisfies (2) in Definition \ref{def_fundamental_family}. As the halfspace $\hh$ is thick, there exists $x\in Y\cap\hh^c$. Any halfspace that separates $x$ from $\pi_{\bar{\hh}}(x)$ separates $x$ from $\bar{\hh}$. Hence, any halfspace in $\HH(x,\pi_{\bar{\hh}}(x))$ contains $\hh$. By the first claim, the family $\HH'$ intersects any interval, particularly $[\pi_{\bar{\hh}}(x),x]$, which completes the proof.
\end{proof}

\begin{rmk}\label{rmk_fundamental_family}
For a complete separable median space of finite rank, we remark the following:
\begin{itemize}
    \item  The fundamental family of halfspaces constructed in Proposition \ref{prop_existence_fundamental_family} contains the subset of atoms of $\HH(X)$, i.e. the subset of halfspaces $\hh\in\HH(X)$ such that $\nu(\{\hh\})>0$. Hence, the subset of atoms in $\HH(X)$ is countable.

    \item If $\Gamma$ is a countable group acting on $X$ by isometries, the fundamental family can be chosen to be $\Gamma$-invariant. In the proof of Proposition \ref{prop_existence_fundamental_family}, one can choose the dense countable subset $Y$ to be $\Gamma$-invariant up to taking $\Gamma.Y$ instead and use the same construction to obtain a $\Gamma$-invariant fundamental family.
\end{itemize}
    
\end{rmk}

\begin{thm}\label{thm_countable_medians_subspace}
    The median subspace $Y'$ arising in Definition \ref{def_median_core} is a countable $Isom(X)$-invariant median subspace.
\end{thm}
\begin{proof}
  Let us fix a fundamental family $\HH'$ and remark that it is enough to prove that $Y_0$ is countable.\par 
Let $\tilde{x},\tilde{y},\tilde{z}\in\partial_r X$ be three distinct points in the regular boundary. There exist then three strongly separated halfspaces $\hh_1,\hh_2,\hh_3$ and $m\in X$ such that $\tilde{x}_i\in\tilde{\hh}_i$ such that for any $\tilde{y}_i\in\tilde{\hh}_i$ we have $m(\tilde{y}_1,\tilde{y}_2,\tilde{y}_3)=m$. By the density of the fundamental family $\HH'$, we can assume that the halfspaces $\hh_1,\hh_2,\hh_3$ are in $\HH'$. Hence, we have a surjection from a subset of $\HH'^3$ onto $Y_0$, implying that the latter is countable which completes the proof. 
\end{proof}

By Theorem \ref{thm_countable_medians_subspace} and Proposition \ref{prop_existence_fundamental_family}, we deduce the following.
\begin{cor}
    If $X$ be a complete separable median space of finite rank, then the restriction of the action of $Isom(X)$ on the median core $Y$ contains an $Isom(X)$-invariant fundamental family. 
\end{cor}
\begin{proof}
    The median subspace $Y$ is taken to be the median closure of the regular boundary inside $X$ ($Y$ is the closure of the countable median subspace $Y'$ constructed in Theorem \ref{thm_countable_medians_subspace}).
\end{proof}
Using the fundamental family, one can easily prove that the regular boundary is measurable.
\begin{prop}\label{prop_regular_boundary_measurable}
    If $X$ is a complete separable median space of finite rank, then the regular boundary is measurable.
\end{prop}

\begin{proof}
 Let us fix a fundamental family $\HH'$, a basepoint $x_0\in X$ and a value $r>0$. For each $p\in\NN$ let $\HH_{x_0,p}^r$ be the subset of halfspaces $\hh\in\HH(X)$ such that there exists halfspaces $\hh_1,..,\hh_p\in\HH(X)$ satisfying:
 \begin{itemize}
     \item $x_0\in \hh_1^c$ and $\hh\subseteq \hh_p$.
     \item $\hh_i^c$ and $\hh_{i+1}$ are strongly separated with $d(\hh_i^c,\hh_{i+1})>r$.
 \end{itemize}
 We set 
 \[ 
 A_p:=\bigcup_{\hh\in\HH_{x_0,p}^r\cap\HH'}\tilde{\hh}. 
 \]
 By Lemma \ref{lem_characterization_regular_points} and the density of the fundamental family, one concludes that 
 \[
 \partial_r X=\bigcap_{p\in\NN_{\geq 1}} A_p,
 \]
which completes the proof.
\end{proof}

\section{G-boundaries and cocycles}\label{section_strongboundaries_cocycles}

\subsection{Cocycles valued in the isometry group of a median space}

We now introduce necessary backgrounds on cocycles, focusing on cocycles which take values in the isometry group of a median space. The main goal is to give a geometric description of the non elementarity condition on the cocycle by the mean of the non elementarity of the induced action on the median space. One can skip this subsection and take Definition \ref{defi_non_elementary_practical_defi} as the definition of the non elementarity. The reader is referred to \cite{BaderFurman}, \cite{Furstenberg-Bourbaki} for related material.\par 
Let $G$ be a locally compact second countable group and let $(\Omega,\mu)$ be a standard probability space endowed with a measurable measure-preserving ergodic $G$-action (also called a $G$-space). When $G=G_1\times G_2$, we say that the action is irreducible if both $G_1$ and $G_2$ act ergodically on $\Omega$.  

\par Let $X$ be a complete median space of finite rank and let 
\[
\alpha:G\times \Omega\rightarrow Isom(X)
\]
be a \textbf{\textit{measurable cocycle}} (over the $G$ action on $(\Omega,\mu)$), that is, $\alpha$ is a measurable map such that for  all $g_1,g_2\in G$ we have $$\alpha(g_1g_2,\omega)=\alpha(g_1,g_2.\omega)\circ \alpha(g_2,\omega)$$ for almost all $\omega\in\Omega$. A cocycle is said to be \textbf{\textit{strict}} if the equality holds for all $g_1,g_2\in G,\ \omega\in \Omega$. By \cite[Appendix B]{Zimmer}, any cocycle $\alpha$ as above coincides almost everywhere with a strict cocycle, hence, we will always assume that cocycles are strict.\par 

Two cocycles $\alpha$ and $\beta$ are said to be \textbf{\textit{cohomologous}} or \textbf{\textit{equivalent}} if there is a measurable map $\phi: \Omega \to Isom(X)$ such that $\forall g \in G, \alpha(g,\omega) = \phi(g.\omega)^{-1}\circ\beta(g,\omega)\circ \phi(\omega)$ for almost all $\omega \in \Omega$.  \par 
 We say that a section $f:\Omega \rightarrow X$ is \textbf{\textit{$\alpha$-equivariant}} if for any $g\in G$ we have 
\[
f(g.\omega)=\alpha(g,\omega). f(\omega)
\]
for almost all $\omega\in\Omega$

In general, Roller minimality and Roller non elementarity are independent conditions. However, if the median space is not lineal, then Roller minimality implies Roller non elementarity (see \cite[Proposition 6.4]{Messaci}). A median space is \textbf{\textit{lineal}} if its Roller compactification is an interval. Note that if a median space is lineal, then any action is Roller elementary. In a non lineal median space the assumption of the flipping lemma with regard to an action is equivalent to the assumption on the action being Roller non elementary and Roller minimal. This motivates the following unified definition of Roller minimality and Roller non elementarity for cocycles (see \cite[Section 3]{Furstenberg-Bourbaki} and \cite[Appendix 7.b]{bader2006superrigidity}).
\begin{defi}\label{defi_non_elementary_practical_defi}
Let $X$ be a complete non lineal median space of finite rank.
    We say that a cocycle $\alpha:G\times\Omega\rightarrow Isom(X)$ is \textit{\textbf{non elementary}} if for any halfspace $\hh\in\HH(X)$ and any subset $A\subseteq \Omega$ of positive measure, there exists $g\in G$ and $\omega\in A \cap g^{-1}.A$ such that $\alpha(g,\omega).\hh\subseteq \hh^c$.
\end{defi}
Let us remark in what follows that in the case where the cocycle arises as the returning cocycle with regard to an action of a lattice, then the non elementarity coincides with the Roller non elementarity and Roller minimality of the action.\par 
Let $\Gamma\leq G$ be a lattice and let us fix a right fundamental domain $\mathfrak{D}\subseteq G$ for $\Gamma$. Endow the fundamental domain with the restriction of the Haar measure of $G$ normalized such that the volume of $\mathfrak{D}$ is one and consider the ergodic $G$-action on it given through the identification between  $\mathfrak{D}$ and $G/\Gamma$. The returning cocycle $ \alpha_{\mathfrak{D}}:G\times \mathfrak{D}\rightarrow\Gamma 
$ associates to each $g\in G$ and $\omega\in\mathfrak{D}$ the unique $\alpha_\mathfrak{D}(g,\omega)\in \Gamma$ such that $g.\omega.(\alpha(g,\omega))^{-1}\in \mathfrak{D}$.
\begin{prop}\label{prop_elementarity_cocycle_vs_homomorphism}
    Let $X$ be an irreducible non lineal median space of finite rank and let $\pi:\Gamma\rightarrow Isom(X)$ be a representation of $\Gamma$.\par 
    The cocycle $\pi\circ \alpha_\mathfrak{D}$ is non elementary if and only if the $\Gamma$-action is Roller minimal.
\end{prop}

\begin{proof}
Let us set $\tilde{\alpha}_\mathfrak{D}:=\pi\circ \alpha_\mathfrak{D}$ and remark that if $\tilde{\alpha}_\mathfrak{D}$ is non elementary, then any thick halfspace is flippable by some element of $\pi(\Gamma)$. \par 
Let us assume now that the $\Gamma$-action is Roller minimal and consider a thick halfspace $\hh\in\HH(X)$ with a subset $A\subseteq \mathfrak{D}$ of positive measure. There exists $\gamma\in \Gamma$ such that $\gamma.\hh\subseteq \hh^c$. As $G.A=G$, there exists $g\in G$ such that $g.A\cap A.\gamma^{-1} \neq \emptyset$. Hence, for any $\omega\in A\cap g^{-1}.A.\gamma^{-1}$, we have $\alpha_{\mathfrak{D}}(g,\omega)=\gamma$ which completes the proof.
\end{proof}

Let us show that elementarity of the cocycle implies the non existence of a fixed point for the induced $G$-twisted action on the space of sections from $\Omega$ to the median space $X$
\begin{prop}\label{prop_no_equivariant_section}
    If the cocycle $\alpha:G\times \Omega \rightarrow Isom(X)$ is non elementary, then there exists no $\alpha$-equivariant section from $\Omega$ to $\bar X$.
\end{prop}
\begin{proof}
    Let us proceed by contradiction and assume that there exists an $\alpha$-equivariant section $f:\Omega\rightarrow \bar X$. By Proposition \ref{prop_Theorem_B5}, there exists a $G$-invariant subset of total measure $\Omega_0\subseteq \Omega$ and a Borel map $$f_0:\Omega_0\rightarrow \bar X $$ that coincides almost everywhere with $f$ and such that for any $g\in G,\ \omega\in \Omega$ we have $f_0(g.\omega)=\alpha(g,\omega).f_0(\omega)$. \par 
    Let $\hh\in\HH(X)$ be a thick halfspace such that $A:=f_0^{-1}(\tilde{\hh})$ is of positive measure. The non elementarity assumption on the cocycle implies that there exist $g\in G$ and $\omega\in A\cap g^{-1}.A$ such that $\alpha(g,\omega).\hh\subseteq \hh^c$, which contradicts the $\alpha$-equivariant of the map.
\end{proof}

Note that in the definition of the $\alpha$-equivariance of a section $f:\Omega\rightarrow X$, the subset of total measure of $\Omega$ on which the equivariance equality holds depends on the element $g\in G$. In some places, it is more convenient to work with a section where such a subset of total measure does not depends on the elements of $G$ and is $G-$invariant. This issue is treated in the appendix of \cite{Zimmer} (Proposition B5) and is written in the context where the target space admit an action and the equivariance is with regard to the actioin of $G$. The same proof works in the context of equivariance by cocycle that we write here for the convenience of the reader.

\begin{prop}[Cocycle restatement of Proposition B5 \cite{Zimmer}]\label{prop_Theorem_B5}
Let $\alpha:\Omega\times G\rightarrow X$ be a cocycle as above and let $f:\Omega\rightarrow X$ be an $\alpha$-equivariant map. There exists then a $G$-invariant measurable subset $\Omega_0\subseteq \Omega$ of total measure and a measurable map $f_0:\Omega_0\rightarrow X$ that coincides almost everywhere with $f$ and such that for any $g\in G,\ \omega\in \Omega$ we have $f_0(g.\omega)=\alpha(g,\omega).f_0(\omega)$
     .
\end{prop}
\begin{proof}
    Set 
    \[
    \Omega_0:=\{\omega\in\Omega\ | \ (\alpha(g,\omega))^{-1}.f(g.\omega)\ \text{is essentially constant in}\ g \}
    \]
    \textit{Claim 1: $\Omega_0$ is $G$-invariant}\par 
    Let us fix $\omega\in \Omega_0$, $h\in G$ and show that $h.\omega\in \Omega_0$. For any $g\in G$, we have 
    \begin{eqnarray*}
        (\alpha(g,h.\omega))^{-1}f(g.h.\omega)&=& (\alpha(g,h.w)\circ\alpha(h,\omega)\circ(\alpha(h,\omega))^{-1})^{-1}.f(gh.\omega)\\
        &=&(\alpha(gh,\omega)\circ(\alpha(h,\omega))^{-1})^{-1}.f(gh.\omega)\\
        &=&(\alpha(h,\omega)\circ \alpha(gh,\omega)^{-1}).f(gh.\omega)\\
        &=&\alpha(h,\omega).(\alpha(gh,\omega)^{-1}.f(gh.\omega))
    \end{eqnarray*}
    which is essentially constant in $g$ and complete the proof of the first claim.\par 
    Define $f_0(\omega)$ to be equal to $(\alpha(g,\omega))^{-1}.f(g.\omega)$ for almost all $g\in G$. \par 
    \noindent\textit{Claim 2: $f_0$ is Borel.} \par Consider a probability measure $\mu$ on $G$ in the same measure class as the Haar measure and remark that 
    \[
    f_0(\omega)=\int_{G}(\alpha(g,\omega))^{-1}.f(g.\omega)d\mu 
    \]
    which is a Borel map by Fubini. (the idea is borrowed from the proof of \cite[Theorem B2]{Zimmer}).\par 
    \noindent\textit{Claim 3: $f_0(g.\omega)=\alpha(g,\omega).f_0(\omega)$ for every $g\in G,\ \omega\in \Omega_0$.}\par 
    Same as in Claim 1, fix $h\in G, \ \omega\in \Omega_0$ and remark that $f_0(h.\omega)$ is defined to be equal $(\alpha(g,h.\omega))^{-1}.f(gh.\omega)$ for almost all $g\in G$. As 
    \[
    (\alpha(g,h.\omega))^{-1}.f(gh.\omega)=\alpha(h,\omega).((\alpha(g.h,\omega)^{-1}.f(gh.\omega))
    \]
    for almost all $g\in G$, we deduce that $f_0(h.\omega)=\alpha(h,\omega).f_0(\omega)$, which completes the claim and finishes the proof.
\end{proof}

\subsection{\texorpdfstring{$G$-boundary}{G-boundary}}
Let $G$ be a locally compact second countable group. In the paper, we use the following definition of $G$-boundary from \cite{BaderFurman}.
\begin{defi}\label{def_G_boundary}
    Let $(B,\mu)$ be a standard probability space endowed with a measure class preserving $G$-action. We say that $(B,\mu)$ is a \textbf{\textit{$G$-boundary}} if 
    \begin{itemize}
        \item The action $G\curvearrowright B$ is \textbf{\textit{amenable}} in the sense of Zimmer.
        \item The diagonal action $G\curvearrowright B\times B$ is \textbf{\textit{isometrically ergodic}}, that is, if $(X,d)$ is a separable metric space endowed with a continuous isometric $G$-action then any $G$-equivariant measurable map $B\times B\rightarrow (X,d)$ is constant.
    \end{itemize}
\end{defi}
\begin{exm} (See \cite{BCFS} or \cite{BaderFurman})
    \begin{itemize}
        \item If $G$ is a finitely generated Gromov hyperbolic group, then its Gromov boundary endowed with the Patterson-Sullivan measure is a $G$-boundary. 
        \item If $G$ is a semisimple Lie group, then $G/P$, where $P$ is a minimal parabolic subgroup, endowed with the unique $G$-invariant measure class, is a $G$-boundary.
        \item If $G$ is any locally compact second countable group and $\mu$ is a symmetric, spread-out generating probability measure on $G$, then its Furstenberg-Poisson boundary is a $G$-boundary.
    \end{itemize}
\end{exm}

In subsection \ref{subsection_construction_boundary_map}, the isometric ergodicity of the action $G\curvearrowright B$ is used to prove the unicity of the boundary map. The following lemma will serve the same purpose in the context of cocycles.
\begin{lem}[See subsection 4.1 \cite{Guirardel-Horbez-Lecureux}]\label{lem_metric_ergodicity_cocycle}
    Let $X$ be a complete separable metric space and let $\alpha:G\times\Omega\rightarrow Isom(X)$. If $(\Theta ,\nu)$ is a standard probability space endowed with a measure class preserving $G$-action that is isometrically ergodic, then any $\alpha$-equivariant section $f:\Omega\times \Theta\rightarrow X$ gives rise to an $\alpha$-equivariant section $\hat{f}:\Omega\rightarrow X$.
\end{lem}


\section{The barycenter map}\label{section_barycenter_map}

 The aim of the section is to show that one can associate measurably a ``center of mass" to each probability measure on the Roller compactification $\bar{X}$ of a median space $X$. It is a essential ingredient in the construction of the boundary map, see Section \ref{section_superrigidity_of_homomorphisms}. The way to characterize the ``center of mass" is through halfspaces. Each halfspace having measure bigger than $\frac{1}{2}$ contains the center of mass. The center of mass is then defined to be the intersection of all closed halfspaces having a measure bigger than $\frac{1}{2}$. Hence, the center of mass is a closed convex subset of $\bar X$. More concretely, let $X$ be a complete separable median space of finite rank on which $Isom(X)$ acts Roller non elementarily and Roller minimally. We denote by $\PP(\bar X)$ the space of probability measures on $\bar X$ endowed with the Borel structure with regard to the weak$^*$ convergence. Using the same notation as in \cite{ChaIF}, we associate to each probability measure $\mu\in\PP(\bar{X})$ the following subsets of halfspaces 
\begin{eqnarray*}
\HH_\mu^+ &:=& \{\hh\in\HH(X)\ | \ \mu(\hh)>\frac{1}{2}\},\\
\HH_\mu^e &:=& \{\hh\in \HH(X)\ | \ \mu(\hh)=\frac{1}{2}\}.
\end{eqnarray*}
\begin{defi}
    
The center of mass of $\mu$ is
$$
C_\mu     := \bigcap_{\substack{\hh\in\HH_\mu^+ \\ \tilde{\hh} \ \text{closed}}} \tilde{\hh}
$$
where $\tilde{h}$ is the canonical halfspace of $\bar X$ associated to $\hh$ (see subsection \ref{subsection_Roller}).
\end{defi} 
Note that, by Helly's property and by the compactness of $\bar X$, the center of mass $C_\mu$ is a non empty closed convex subset. \par 

The subset of measures that will be of interest are those that have a single point as a center of mass. We denote this subset of measures by
\[
\PP_1(\bar X):=\{\mu\in\PP(\bar X)\ |\ C_\mu\ \text{is a singleton}\ \}
\] 
We denote by $\PP_2(\bar X)$ the complement of $\PP_1(\bar X)$ in $\PP(\bar{X})$, that is, the subset of measures having a non singleton barycenter.\par 
In the next sections, we will show that the boundary map will take image in fact inside $\PP_1(\bar X)$. To do so, we need the following criterion to determine when $C_{\mu}$ is a singleton for $\mu \in \PP(\bar{X})$.
\begin{prop}\label{prop_interval_balanced_halfspaces}
    Let $\mu \in \PP(\bar{X})$. The closed convex subset $C_\mu$ is a singleton if and only if $\HH_\mu^e$ contains no halfspace-intervals, that is, there is no pair $x \neq y$ in $X$ such that $\HH(x,y)\subseteq \HH_\mu^+$.
\end{prop}
Let us separately prove the first direction of Proposition \ref{prop_interval_balanced_halfspaces} in the following lemma.

\begin{lem}\label{lem_balanced_halfspace_transverse}
    For any $\tilde{x},\tilde{y}\in C_\mu$, we have $\mathring\HH(\tilde{x},\tilde{y})\subseteq \HH_\mu^e$.\par 
   In particular, Lemma \ref{lem_separation_Roller_compactification} shows that there exists $x,y\in X$ such that $\HH(x,y)\subseteq \HH_\mu^e$.
\end{lem}
\begin{proof}
    By Lemma \ref{lem_separation_Roller_compactification}, any halfspace in $\mathring \HH(\tilde{x},\tilde{y})$ is of the form $\tilde{\hh}$ where $\hh\in\mathring\HH(x,y)$ for some $x,y\in X$. Let $\hh$ be such halfspace. By \cite[Corollary 2.23]{Fior_median_property}, at least one of $\hh$ and $\hh^c$ is closed. Hence, by the construction of $C_\mu$, there is no closed halfspace in $\HH_\mu^+$ that is transverse to it. Thus, both $\hh$ and $\hh^c$ belong to $\HH_\mu^e$, which completes the proof.
\end{proof}

\begin{proof}[Proof of Proposition \ref{prop_interval_balanced_halfspaces}]
Let us assume that $C_\mu$ is a singleton $\{p\}$ and consider $\tilde{x} \neq \tilde{y}\in \bar X$. We set $\tilde{p}:=m(p,\tilde{x},\tilde{y})$ and assume without loss of generality that $\tilde{p}\neq \tilde{x}$. Let us consider a closed halfspace $\hh\in \mathring\HH(\tilde{p},\tilde{x})$ and set $p_\hh=\pi_{\hh}(p) (\neq p)$. Since the closed convex subset $C_\mu$ consisting of the singleton $p$, there exists by definition a closed halfspace $\hh'\in\HH_\mu^+$ separating $p$ from $p_\hh$. Hence, we get $\hh'\subseteq \hh^c$, which implies that $\hh$ does not lie in $\HH_\mu^e$, that is, $\HH(\tilde{x},\tilde{y})$ is not contained in $\HH_\mu^e$.\par

\par    The converse is given by Lemma \ref{lem_balanced_halfspace_transverse}.
\end{proof}

Fixing a fundamental familly of halfspaces $\HH'$, the following lemma shows that, for any measure $\mu\in \PP(\bar X)$, there exists a halfspace in the fundamental family that contains almost all part of the support of $\mu$.
\begin{lem}\label{lem_measures_having_consistant_support_in_halfspace}
    Let $X$ be an irreducible complete separable median space of finite rank and $\HH'$ the dense family of halfspaces as in Proposition \ref{prop_existence_fundamental_family}. Then for any $\varepsilon>0$, we have 
    \[
    \PP(\bar{X})\subseteq \bigcup_{\hh\in\HH'}ev_\hh^{-1}(]1-\varepsilon,1])
    \]
    where 
    \begin{eqnarray*} 
    ev_\hh:\PP(\bar{X})&\rightarrow & [0,1] \\
           \mu &\mapsto& \mu(\hh)
       \end{eqnarray*}
\end{lem}
\begin{proof}
    As $X$ is irreducible and the action of $Isom(X)$ is Roller minimal and Roller non elementary, the regular boundary $\partial_r X$ is not empty and contains infinitely many elements. Let $\epsilon > 0$ and fix a measure $\mu\in\PP(\bar{X})$ and consider any point $\tilde{x}\in\partial_r X$ such that $\mu(\{\tilde{x}\})<\varepsilon$. The point $\tilde{x}$ being in $\partial_r X$, there exists a descending sequence of thick halfspaces $\{\hh_i\}_{i\in\NN}$ such that $\hh_i^c$ and $\hh_{i+1}$ are strongly separated with $\displaystyle{\bigcap_{i\in\NN}\tilde{\hh_i}}=\{\tilde{x}\}$. Hence, there exists $n\in\NN$ such that $\mu(\hh_n)<\varepsilon$. 
    Thus, by Proposition \ref{prop_existence_fundamental_family}, there exists a halfspace $\hh'\in \HH'$ such that $\hh_n^c\subseteq\hh'$ which implies that $\mu(\hh')>1-\epsilon$ and finishes the proof. 
\end{proof}

\begin{rmk}\label{rmk_measures_having_consistant_support_in_halfspace}
    Note that for any measure $\mu$ and for any thick halfspace $\hh\in\HH(X)$. The halfspace $\hh'\in\HH'$ obtained from the proof of Lemma \ref{lem_measures_having_consistant_support_in_halfspace} that verifies $\mu(\hh')>1-\varepsilon$ can be chosen such that $\hh^c\subseteq \hh'$. This comes from the fact that any thick halfspace contains infinitely many points from the regular boundary. Hence, the argument given in the proof of Lemma \ref{lem_measures_having_consistant_support_in_halfspace} can be made by considering regular points lying inside $\tilde{\hh}$.
\end{rmk}

\begin{lem}\label{lem_countable_balanced_halfspaces}
  If $X$ is a median space of finite rank, and let $\HH'\subseteq \HH(X)$ be a fundamental family. Then we have
  \[
  \PP_1(\bar X)^c\subseteq \bigcup_{\hh\in\HH'}ev_\hh^{-1}(\{\frac{1}{2}\})  
 \]
\end{lem}
\begin{proof}
    Let us fix $\mu\in \PP_1(\bar X)^c$ and show that it lies inside $ev_\hh^{-1}(\{\frac{1}{2}\}) $ for some halfspace $\hh\in\HH'$. As $C_\mu$ is distinct from a point, Proposition \ref{prop_interval_balanced_halfspaces} ensures the existence of $x,y\in X$ such that $\HH(x,y)\subseteq \HH_\mu^e$. By Proposition \ref{prop_existence_fundamental_family}, there exists a halfspace $\hh\in\HH'$ separating $x$ and $y$. In particular, $\hh\in\HH_\mu^e$ that is, $\mu(\hh)=\frac{1}{2}$. This is equivalent to say that $\mu\in ev_\hh^{-1}(\{\frac{1}{2}\})$, which proves the lemma.
\end{proof}
The remaining part of this section is that the map that associates to each measure $\mu\in \PP_1(\bar X)$ its barycenter $C_{\mu}$ in $\bar X$ is measurable, here we regard $C_{\mu}$ as a point in $\bar{X}$. 
\begin{prop}\label{prop_measurability_barycenter_map}
If $X$ is a complete separable median space of finite rank, then the $Isom(X)$-equivariant map 
\begin{eqnarray*}
\bb:\PP_1(\bar X)&\rightarrow& \bar X\\
\mu &\mapsto& C_\mu
\end{eqnarray*}
is measurable.
\end{prop}
Before proving Proposition \ref{prop_measurability_barycenter_map}, we first prove the following lemma.
\begin{lem}\label{lem_closed_halfspace_close_to_open_halfspace}
    Let $Y$ be a complete separable median space of finite rank and $\mu\in\PP(\bar{Y})$ a probability measure. Then for any open halfspace $\hh_o\in \HH(Y)$ and any $\varepsilon>0$, there exists a closed halfspace $\hh\subseteq \hh_0$ such that 
    \[
    \mu(\tilde{\hh})>\mu(\tilde{\hh}_o)-\varepsilon . 
    \]
\end{lem}
\begin{proof}
  Let $\hh_o\subset Y$ be an open halfspace. For every $n\in\NN_{\geq 1}$, we set:
  \[
  \HH_{\hh_o,n}:=\{\hh\in \HH(Y)\ | \ d(\hh,\hh_o^c)\geq\frac{1}{n}\ \text{with}\ \hh\  \text{open and maximal}\}
  \]
As $Y$ is separable, for each $n$, the subset $\HH_{\hh_o,n}$ is not empty and countable (see Proposition \ref{prop_minimal_halfspace_countable}). \par
We first claim that for any $\tilde{x}\in \tilde{\hh}_0$, there exists $k\in\NN$ and $\hh' \in \HH_{\hh_o,k}$ such that $\tilde{x}\in\tilde{\hh}'$. Indeed, any pair of distinct points in $\bar{X}$ is separated by a halfspace $\tilde{\hh}\subseteq \bar{X}$ for some $\hh\in \HH(X)$. Hence, let us consider a closed halfspace $\hh\in\HH(X)$ that separates $\tilde{x}$ from $\pi_{\tilde{\hh}_0^c}(\tilde{x})$. For any $k\in\NN$ such that $\frac{1}{k}<d(\hh,\hh_o^c)$, there exists $\hh'\in\HH_{\hh_o,k}$ that contains $\hh$ (note that it can $\hh$ itself). Thus we get $\tilde{x}\in\tilde{\hh}'$.\par 
For each $n\in\NN_{\geq 1}$, we set
\[
C_n:=\bigcup_{\hh\in\HH_{\hh_0,n}}\tilde{\hh}.
\]
Note that we have $C_n\subseteq C_{n+1}$, as any halfspace in $C_n$ is contained in a halfspace in $C_{n+1}$.\par
We have $\hh_o=\displaystyle{\bigcup_{n\in\NN_{\geq 1}} C_n}$ by the claim. Hence, there exists $n\in\NN_{\geq 1}$ such that 
\[
\mu(C_n)>\mu(\tilde{\hh}_o)-\frac{1}{2}\varepsilon.
\]
Since $\HH_{\hh_0,n}$ is countable, consider then $\hh_1,...,\hh_l\in C_n$ such that 
\[
\mu(\tilde{\hh}_1\cup..\cup\tilde{\hh}_l)>\mu(C_n)-\frac{1}{2}\varepsilon.
\]
As the convex hull between finitely many closed convex subsets is also closed, we deduce that $C:=Conv(\bar {\tilde{\hh}}_1\cup ...\cup\bar{\tilde{\hh}}_l)$ is a closed convex subset contained in $\hh_o$. As both $C$ and $\tilde{\hh}_o^c$ are disjoint closed convex subsets, there exists a closed halfspace $\hh\in\HH(X)$ such that $\tilde{\hh}$ separates $C$ from $\tilde{\hh}_o^c$. We conclude that 
\[
\mu(\tilde{\hh}_c)>\mu(\tilde{\hh}_o)-\varepsilon
\]
which completes the proof.
\end{proof}
\begin{proof}[Proof of Proposition \ref{prop_measurability_barycenter_map}]
     By Proposition \ref{prop_Borel_structure_Roller}, to show that the map $\bb$ is measurable, it is enough to show that any open halfspace under the map $\bb$ is a measurable subset of $\PP_1(X)$. \par For any open halfspace $\hh\in\HH(X)$, we claim that
    \[
    \bb^{-1}(\hh)=ev_\hh^{-1}(]\frac{1}{2},1])
    \]

    Let us first show that $ \bb^{-1}(\hh) \subseteq ev_\hh^{-1}(]\frac{1}{2},1])$. Take $\mu\in\bb^{-1}(\hh)$. As $\bb(\mu) = C_\mu$ consists of a singleton and $d(\bb(\mu),\pi_{\hh^c}(\bb(\mu)))=d(\bb(\mu),\hh^c)>0$, there exists a closed halfspace $\hh'\in\HH_\mu^+$ that separates $\bb(\mu)$ from $\pi_{\hh^c}(\bb(\mu))$. Hence, we get $\hh'\subseteq \hh$, which implies that $\bb(\mu)(\hh)>\frac{1}{2}$, that is, $\mu\in ev_\hh^{-1}(]\frac{1}{2},1])$. \\
    To show the converse inclusion, let us fix $\mu\in ev_\hh^{-1}(]\frac{1}{2},1])$.  By Lemma \ref{lem_closed_halfspace_close_to_open_halfspace}, there exists a closed halfspace $\hh'\subseteq \hh$ such that $\bb(\mu)(\tilde{\hh}')>\frac{1}{2}$. This implies that $\hh'\in \HH_\mu^+$, hence $\bb(\mu)\in\hh'\subseteq \hh$, which completes the proof of the claim. Hence, $\bb$ is measurable.
\end{proof}

\section{Superrigidity of homomorphisms}\label{section_superrigidity_of_homomorphisms}

\subsection{Construction of the boundary map}\label{subsection_construction_boundary_map}

Throughout this section, we denote by $\Gamma$ a finitely generated group and by $B$ be a $\Gamma$-boundary. The reader is referred to Section \ref{section_strongboundaries_cocycles} for the background on related facts about boundaries.\par  Let $X$ be a complete, irreducible, separable median space of finite rank, and assume that $\Gamma$ acts Roller non-elementarily and Roller minimally on $X$. Hence, $\partial_r X$ is in particular uncountable (see Remark \ref{rmk_construction_regular_point}). Up to restricting to the median closure of a $\Gamma$-orbit, there is no loss of generality to assume $X$ is separable. By fixing a $\Gamma$-invariant countable dense subset, we consider a $\Gamma-$invariant fundamental family $\HH'$ (see the second point in Remark \ref{rmk_fundamental_family}). We remark that in this section, we only require the fundamental family to be invariant under $\Gamma$ and not necessarily under the whole $Isom(X)$, hence there is no need to restrict to the median core.
As the action of $\Gamma$ on $B$ is amenable, there exists an equivariant map $\Phi:B\rightarrow \PP(\bar{X})$ where $\PP(\bar{X})$ is the set of probability measures on $\bar{X}$. The first step towards superrigidity of homomorphims is to derive a $\Gamma$-equivariant map from $B$ to $\bar{X}$. To do so, we associate to each measure $\Phi(B)$ its barycenter defined in Section \ref{section_barycenter_map}. If $\Phi(B) \subset \PP_1(X)$, where $\PP_1(X)$, as defined in Section \ref{section_barycenter_map}, is the set of probability measures on $\bar{X}$ whose centers of mass are single points, then Proposition \ref{prop_measurability_barycenter_map} indicates that we have a map from $B$ to $\bar{X}$. The issue however is that one may have that the center of mass of $\Phi(b)$ is not a singleton for some $b \in B$. Nevertheless, it is already sufficient for us that for almost all points $b$, the barycenter of $\Phi(b)$ is a single point.  \par


\par
Namely, by Helly's theorem and by the compactness of the space $\bar{X}$, $C_\mu \subseteq \bar{X}$ is a non empty closed convex subset.  For convenience, we simply write $\HH_b^+$, $\HH_b^e$, and $C_b$ for $\HH_{\Phi(b)}^+$, $\HH_{\Phi(b)}^e$ and $C_{\Phi(b)}$, respectively, where $b\in B$. We now show that, for almost all points $b \in B$, $C_b = C_{\Phi(b)}$ is a singleton. For that, we need first to fix a fundamental family of halfspaces $\HH'$, which is a $\Gamma$-invariant countable ``dense" subset of halfspaces as in Proposition \ref{prop_existence_fundamental_family}

\begin{prop}\label{prop_balanced_halfspaces_zero_measure}
The subset of $B_e:=\{b\in B\ |\ C_b\ \text{is not a singleton} \} $ is contained in a subset of zero $\nu$-measures.
\end{prop}
Note that one can assume that $B_e$ is measurable since we can always assume that all measures we considered are complete.

\begin{proof}
        Consider the following map
        \begin{eqnarray*}
            \psi: B\times B &\rightarrow& \RR\\
            \ \ \ (b_1,b_2)&\mapsto& \inf_{h\in\HH'}{\{|ev_\hh(\Phi(b_1))-ev_{\hh^c}(\Phi(b_1))|+|ev_\hh(\Phi(b_2))-ev_{\hh^c}(\Phi(b_2))|\}}
        \end{eqnarray*}
        The map $\psi$ is measurable and $\Gamma$-invariant. \par 
        \textbf{Claim:} The map $\psi(b_1,b_2)$ has a value greater than or equal to $1$ for almost all $(b_1,b_2)\in B\times B$. \par 
        To prove the claim, fix a small enough $\varepsilon>0$ and choose $\hh\in\HH'$ such that $A:=\Phi^{-1}(ev_\hh^{-1}(]1-\varepsilon, 1]))$ is a measurable subset of positive measure in $B$, this is ensured by Lemma \ref{lem_measures_having_consistant_support_in_halfspace}. Consider two strongly separated thick halfspaces $\hh_1,\hh_2\in X$ and two isometries $g_1,g_2\in \Gamma$ such that $g_i(\hh)\subseteq \hh_i$. We set $A_1:=g_1.A$ and $A_2:=g_2.A$. Thus, for any $b_1\in A_1$ and $b_2\in A_2$ and for any halfspace $\hh' \in\HH(X)$ which is not transversal with $\hh_1$, we have either $\hh_1\subseteq \hh'$ or $\hh'\subseteq \hh_1$ in which case we get $\hh_2\subseteq \hh_1^c\subseteq \hh'^c$. Hence, we have either $ev_{\hh'}(\Phi(b_1))>1-\varepsilon$ or $ev_{(\hh')^c}(\Phi(b_2))>1-\varepsilon$. In both the cases, we get 
        \[
        |ev_{\hh'}(\Phi(b_1))-ev_{(\hh')^c}(\Phi(b_1))|+|ev_{\hh'}(\Phi(b_2))-ev_{(\hh')^c}(\Phi(b_2))|>1-2\varepsilon.
        \]
        The same conclusion holds if $\hh'$ is not transversal to $\hh_2$.
        \par 
        Since the two halfspaces $\hh_1,\hh_2$ are strongly separated, any halfspace $\hh\in \HH(X) $ is not transversal with at least one of them. As a consequence, we have $\psi(b_1,b_2)>1-2\varepsilon$ for any $(b_1,b_2)\in A_1\times A_2$. Since the subset $A_1\times A_2 \subset B\times B$ is a measurable subset of positive measure, so is $\Gamma.(A_1\times A_2)$ which is moreover $\Gamma$-invariant. Hence, by the ergodicity of the action of $\Gamma$ on $B\times B$ we deduce that for almost all $(b_1,b_2)\in B\times B$ we have $\psi(b_1,b_2)>1-2\varepsilon$. Hence, for any $\varepsilon > 0$, the subset $\psi^{-1}(]1-\varepsilon,2])$ is of total measure in $B\times B$. By considering the intersection $\displaystyle{\bigcap_{n\in\NN_{\geq 1}}\psi^{-1}(]1-\frac{1}{n},2])}$, we deduce that for almost all $b_1,b_2\in B$, we have $\psi(b_1,b_2)\geq 1$ which finishes the proof of the claim.
        
        \par
        We now use the claim to finish the proof of the proposition. Note that, by Lemma \ref{lem_countable_balanced_halfspaces} and $\HH'$ being countable, it is enough to prove that each $ev_\hh^{-1}(\{\frac{1}{2}\})$ is of zero measure to complete the proof of the proposition. Remark that for any $b_1,b_2\in ev_\hh^{-1}(\{\frac{1}{2}\})$, we have $\Psi(b_1,b_2)=0$. The map $\psi$ being greater or equal to $1$ almost everywhere, we deduce that $ev_\hh^{-1}(\{\frac{1}{2}\})$ is of zero measure, which finishes the proof of the proposition.
    \end{proof}
Proposition \ref{prop_balanced_halfspaces_zero_measure} allows us to get a boundary map from $B$ to the Roller boundary $\partial X$.
\begin{thm}\label{thm_contruction_boundary_map}
 There exists a measurable $\Gamma$-equivariant map $\phi:B\rightarrow \partial{X}$.
\end{thm}
\begin{proof}
    By Proposition \ref{prop_balanced_halfspaces_zero_measure}, we have a $\Gamma$-equivariant map 
    \[
     \phi:B\rightarrow \bar{X}
    \]
    defined almost everywhere that maps each $b$ to the singleton $C_b$, that is, the intersection of all closed halfspaces $\hh$ such that $\Phi(b)(\hh)>\frac{1}{2}$. \par 
    \par
    The decomposition $\bar{X}:=X\sqcup \partial X$ is measurable and $\Gamma$-invariant. By Proposition \ref{prop_Borel_structure_Roller} and the metric ergodicity of the $\Gamma$-action on $B$, we deduce that the image of $\phi$ lies essentially in $\partial X$.
\end{proof}

\begin{lem}\label{lem_inverse_image_halfspace_positive_measure}
   For any thick halfspace $\hh\in\HH(X)$, the subset $\phi^{-1}(\tilde{\hh})$ is of positive measure in $B$.
    \end{lem}
    \begin{proof}
   Lemma \ref{lem_measures_having_consistant_support_in_halfspace} yields the existence of halfspace $\hh'\in\HH'$ such that the subset 
   \[
   A_{\hh'}:=\Phi^{-1}(ev_{\hh'}(]1-\varepsilon,1])=\{b\in B\ |\ \Phi(b)(\hh')>1-\varepsilon\}
   \]
   is of positive measure. For any thick halfspace $\hh$, there exists $g\in\Gamma$ such that $g.\hh'\subseteq \hh$. Therefore, we get $g.A_{\hh'}\subseteq A_\hh$. By taking $\varepsilon$ small enough, we conclude that $\hh\in\HH_b^+$ for any $b\in A_\hh$, implying that $\phi(b)\in \hh$ for any $b\in A_\hh$, and therefore, $A_\hh\subseteq \phi^{-1}(\hh)$ which completes the proof.
    \end{proof}
    Actually, to get superrigidity, it is crucial to us that the image of the boundary map $\phi$ only consists of regular points. Namely, 
\begin{thm}\label{thm_image_in_regular_boundary}
Let $\phi$ be the boundary map given by Theorem \ref{thm_contruction_boundary_map}. The image of $\phi$ lies essentially in the regular boundary $\partial_r X$.
\end{thm}
\begin{proof}
In the first step, we claim that for almost all $b_1,b_2\in B\times B$, their images under $\phi$ are separated by arbitrary large sequence of strongly separated halfspaces, that is, for some fixed $r>0$ and for any $n\in \NN$, there exist halfspaces $\hh_1,\hh_2,...,\hh_n\in\HH(X)$ such that:
\begin{itemize}
    \item $\phi(b_1)\in\tilde{\hh}_1^c$ and $\phi(b_2)\in \tilde{\hh}_n$.
    \item $\hh_i^c$ and $\hh_{i+1}$ are strongly separated with $d(\hh_i^c,\hh_{i+1})>r$.
\end{itemize}
Indeed, let us denote by $B_n^r\subseteq B\times B$ the subset of pair $b_1,b_2$ for which the above family of halfspaces exists and show that it contains a measurable subset of total measure. The action of $\Gamma$ on $X$ being Roller minimal and Roller non elementary and the space $X$ being irreducible, there exist halfspaces $\hh_1,..,\hh_n\in \HH(X)$ such that $\hh_i^c$ and $\hh_{i+1}$ are strongly separated with $d(\hh_i^c,\hh_{i+1})>r$ for all $i\in\{1,...,n-1\}$. By Lemma \ref{lem_inverse_image_halfspace_positive_measure}, the subsets $A_1:=\phi^{-1}(\tilde{\hh}_1^c)$ and $A_2:=\phi^{-1}(\tilde{\hh}_n)$ are measurable subsets of positive measure in $B$. Hence, the subset $\Gamma.(A_1\times A_2)\subseteq B_n^r$ is a $\Gamma$-invariant measurable subset of positive measure. By the ergodicity of the $\Gamma$-action on $B\times B$, we deduce that it is of total measure, which proves the claim.\par 
We set $\tilde{B}:=\displaystyle{\bigcap_{n\in\NN_{\geq 1}}B_n^r}\subseteq B\times B$ and fix a basepoint $x_0\in X$. So $\tilde{B}$, being a countable intersection of full measure sets, has a total measure. For $(b_1,b_2)\in\tilde{B}$ and $n\in\NN$, there exists a sequence of halfspaces $\hh_1,...,\hh_n$ that strongly separates $\phi(b_1)$ and $\phi(b_2)$ as described above. We have either $x_0\in\hh_{[\frac{n}{2}]}$ or $x_0\in\hh_{[\frac{n}{2}]}^c$. Hence, for every $(b_1,b_2)\in \tilde{B}$ we have at least one of $\phi(b_1)$ and $\phi(b_2)$ which falls under the hypothesis of Lemma \ref{lem_characterization_regular_points}. Hence, at least one of them lies in the regular boundary $\partial_r X$. The regular boundary $\partial_r X$ being measurable by Proposition \ref{prop_regular_boundary_measurable}, we deduce that $(\phi^{-1}(\partial_r X)\times B)\cup(B\times \phi^{-1}(\partial_r X)) $, containing $\tilde{B}$, is a subset of full measure in $B\times B$. Thus, the subset $\phi^{-1}(\partial_r X)$ has positive measure in $B$. As it is a $\Gamma$-invariant subset, we conclude by the ergodicity of the $\Gamma$-action on $B$ that it is a subset of full measure.
\end{proof}

\begin{rmk}
    The proof shows that the subset of $B\times B$ consisting of pairs $b_1,b_2\in B$ such that $\phi(b_1)\neq\phi(b_2)$ is of full measure. 
\end{rmk}
\begin{thm}[Compare with Theorem 3.1 (ii) \cite{BCFS}]\label{thm_unicity_of_boundary_map_product}
Any $\Gamma$-equivariant measurable map $\psi: B\times B \rightarrow \partial_r X$ is essentially of the form $\phi\circ \pi$, where $\pi$ is the projection onto one of the factors.
\end{thm}
\begin{proof}
    Let us consider a $\Gamma$-equivariant measurable map $\psi: B\times B \rightarrow \partial_r X$, and define 
    \begin{eqnarray*}
\Psi: B\times B &\rightarrow & \bar{X}\\
\ \ \ \ (b_1,b_2) &\mapsto & m(\phi(b_1),\phi(b_2),\psi(b_1,b_2))
\end{eqnarray*}
    By Theorem \ref{thm_image_in_regular_boundary}, the image of $\phi$ is essentially in the regular boundary $\partial_r X$. Hence, by Lemma \ref{lem_interval_between_strongly_separated_ultrafilters}, for essentially any $b_1,b_2\in B$ such that $\phi(b_1)\neq \phi(b_2)$, the image $\Psi(b_1,b_2)$ lies in $X$ if and only if $\psi(b_1,b_2)$ is not equal to $\phi(b_1)$ and $\phi(b_2)$. By the metric ergodicity of $\Gamma \curvearrowright B \times B$, the image of $\Psi$ cannot lie in $X$. We deduce that $\psi(b_1,b_2)$ is equal to either $\phi(b_1)$ or $\phi(b_2)$ for almost all pairs $(b_1,b_2)\in B \times B$. Again, by the ergodicty of the $\Gamma$-action on $B\times B$, we deduce that the map $\psi$ is essentially either $\phi\circ\pi_1$ or $\phi\circ\pi_2$, where $\pi_i: B \times B \to B$ is the projection to the $i$th factor.
 \end{proof}

\begin{rmk}
    In contrast with \cite{BCFS}, the argument for the proof of Theorem \ref{thm_unicity_of_boundary_map_product} runs smoothly with the median structure. For any pairwise distinct triple in the regular boundary, one associates to them a unique point (their median point) inside the space, see Lemma \ref{lem_interval_between_strongly_separated_ultrafilters}. Hence, one may apply directly the metric ergodicity to deduce the result. In the case of hyperbolic space, for any pairwise distinct triple of points in the Gromov boundary, there is no canonical point associated to them. The authors of \cite{BCFS} overcame this issue by associating a bounded subset of the hyperbolic space and then they used this correspondence to pull back the Hausdorff distance on the space of bounded subset to the space of distinct triple to endow the latter with a pseudo-metric and use the coarse metric ergodicity, which they introduced therein, to deduce the result. 
\end{rmk}
 \begin{cor}\label{cor_boundary_map_into_product}
Any $\Gamma$-equivariant measurable map $\psi: B\times B\rightarrow (\partial_r X)^2$ is essentially of the following form: 
\begin{enumerate}
 \item $\psi(b_1,b_2)=(\phi(b_1),\phi(b_2))$ (or $\psi(b_1,b_2)=(\phi(b_2),\phi(b_1))$).
 \item $\psi(b_1,b_2)=(\phi(b_i),\phi(b_i))$ where $i\in\{1,2\}$ is fixed.
\end{enumerate}
\end{cor}

\subsection{Superrigidity}

In this section, we prove the superrigidity result for irreducible lattices in a product of locally compact second countable groups. We only state the result where there are only two factors. Let $G_1, G_2$ be two locally compact second countable groups, and let $X$ be an irreducible complete separable median space of finite rank. Let us assume that $\Gamma\leq G_1\times G_2$ is an irreducible lattice of $G_1 \times G_2$ that acts Roller non-elementarily and Roller minimally on $X$. The $G_1\times G_2$-boundary is then given by the product $B_1\times B_2$ where $B_i$ is a $G_i$-boundary. Following \cite{BaderFurman} or \cite{BCFS}, we first show that
\begin{thm}\label{thm_factor_boundary_map}
    The boundary map $\phi: B=B_1\times B_2 \rightarrow \partial_r X$ factors through the projection onto one of the $B_i$'s.
\end{thm}
\begin{proof}
    The proof follows the same line as in \cite{BaderFurman} or \cite{BCFS}. Consider the following measurable map:
    \begin{eqnarray*}
    \eta: B \times B = B_1\times B_2 \times B_1 \times B_2 &\rightarrow & (\partial_r X)^2\\
\ \ \ \ (x,y,x',y') &\mapsto & (\phi(x,y),\phi(x,y'))
    \end{eqnarray*}
 By Corollary \ref{cor_boundary_map_into_product}, we have three cases:
 \begin{enumerate}
     \item $\eta(x,y,x',y')=(\phi(x,y),\phi(x',y'))$ implying that for almost $x,x'\in B_1$ and $y' \in B_2$, we have $\phi(x,y')=\phi(x',y')$. Hence, the map $\phi$ depends only on the factor $B_1$.
     \item $\eta(x,y,x',y')=(\phi(x',y'),\phi(x,y))$ which cannot happen as the map $\phi$ is not constant by the Roller nonelementarity of the $\Gamma$-action on $X$.
     \item $\eta(x,y,x',y')=(\phi(x,y),\phi(x,y))$ which implies that $\phi(x,y)=\phi(x,y')$ for almost all $x,y,y'$. Hence, the function $\phi$ is independent of the factor $B_2$, which finishes the proof.
 \end{enumerate}
 
\end{proof}


As in the construction of the median core, we set 
  \[
    Y_0^\phi:=\{m_{\bar X}(\tilde{x},\tilde{y},\tilde{z}) \ |\ \tilde{x},\tilde{y},\tilde{z}\in \phi(B)\subseteq \partial_r X, \ x_i\neq x_j\}\subseteq X
  \]
and we define recursively $Y_{i+1}^\phi:= m(Y_i^\phi,Y_i^\phi,Y_i^\phi)$. We set $Y^\phi:=\displaystyle{\overline{\bigcup_{i\in\NN}Y_i^\phi}}$ and remark that it is a $\Gamma$-invariant complete separable median subspace. Hence, this gives rise to a morphism 
\[
i:\Gamma\rightarrow Isom(Y^\phi).
\]
Let us first remark the following.

\begin{prop}
   The median subspace $Y^\phi$ coincides with the median core $Y$ of $X$.
\end{prop}
\begin{proof}
    It is enough to show that $Y_0=Y_0^\phi$. Fix a point $x\in Y_0$ and $\tilde{x}_1,\tilde{x}_2,\tilde{x}_3\in\partial_r X$ such that $m(\tilde{x}_1,\tilde{x}_2,\tilde{x}_3)=x$. There exists then a triple of strongly separated thick halfspaces $\hh_1,\hh_2,\hh_3\in\HH(X)$ such that $\tilde{x}_i\in\tilde{\hh}_i$ and for any $\tilde{y}_i\in\tilde{\hh}_i$ we have $m(\tilde{y}_1,\tilde{y}_2,\tilde{y}_3)=x$. As the $\Gamma$-action on $X$ is Roller non elementary Roller minimal and the boundary map is $\Gamma$-equivariant, there exists $\phi^{-1}(\tilde{\hh}_i\cap\partial_r X)$ is of positive measure for each $i\in\{1,2,3\}$. Hence, for any $b_i\in \phi^{-1}(\tilde{\hh}_i\cap\partial_r X)$, we have $m(\phi(b_1),\phi(b_2),\phi(b_3))=x$ which proves $Y_0=Y_0^\phi$.
\end{proof}
\color{black}
Now, we have all what we need to prove the superrigidity results. Theorem \ref{thm_factor_boundary_map} says that the $\Gamma$-action on the image of the boundary map factors through one factor. It is left to show that the action extends to the median closure of the image of the boundary map using the density of the projection of $\Gamma$ and a topological criterion to ensure that the morphism extends to the topological closure (see \cite[Proposition 4.3]{Shalom-superrigidity}).
    \begin{thm}\label{thm_superrigidity_principal}
        The morphism $i$ extends continuously to $G$ and factors through one of the $G_i$'s. 
    \end{thm}
    \begin{proof}
        By Theorem \ref{thm_factor_boundary_map}, the boundary map $\phi$ factors through the projection into a factor, say $B_1$. We are going to show that the $\Gamma$-action on $Y$ extends continuously to $G_1$. As $Y$ is separable, $Isom(Y)$ is a Polish group when endowed with the pointwise convergence topology. Hence, to extend the action to $G_1$ it is enough to show that the morphism $i$ is continuous when $\Gamma$ is endowed with the induced topology from its projection into $G_1$. The group $G_1$ being second countable, it is enough to show that for any sequence $\{\gamma_n\}_{n\in\NN}\subseteq \Gamma$ such that $\lim_{n \to +\infty} \pi_1(\gamma_n)=e_1$, then $\lim_{n \to +\infty} \gamma_n(x)=x$ for all $x\in Y$. By the continuity of the ternary operation, the commutativity of the action with the ternary operation and the construction of $Y$, it is enough to prove it only for $x\in Y_0$. \par 
        Let us consider then $x_0\in Y_0$, where $x_0=m_{\bar X}(\tilde{x},\tilde{y},\tilde{z})$ for some $\tilde{x},\tilde{y},\tilde{z}\in\phi(B)\subseteq \partial_r X$. The triple $\tilde{x},\tilde{y},\tilde{z}$ being in the regular boundary, there exists a sequence of halfspaces $(\hh_i^{\tilde{x}})_{i\in\NN}$, $(\hh_i^{\tilde{y}})_{i\in\NN}$ and $(\hh_i^{\tilde{z}})_{i\in\NN}$ such that 
        \[
        \tilde{x}=\bigcap_{i\in\NN}\tilde{\hh}_i^{\tilde{x}}, \ \tilde{y}=\bigcap_{i\in\NN}\tilde{\hh}_i^{\tilde{y}},\ \text{and}\ \tilde{z}=\bigcap_{i\in\NN}\tilde{\hh}_i^{\tilde{z}}.
        \]
        For $n\in \NN$ large enough, we can assume that $\hh_n^{\tilde{x}},\hh_n^{\tilde{y}}$ and $\hh_n^{\tilde{z}}$ are strongly separated and for any $x\in\hh_n^{\tilde{x}}$,$y\in\hh_n^{\tilde{y}}$ and $z\in\hh_n^{\tilde{z}}$, we have $m(x,y,z)=x_0$. Fix the following subsets of $B_1$ 
        \[
        A_{\tilde{x}}=\pi_1(\phi^{-1}(\hh_n^{\tilde{x}})),\ A_{\tilde{y}}=\pi_1(\phi^{-1}(\hh_n^{\tilde{y}})),\ A_{\tilde{z}}=\pi_1(\phi^{-1}(\hh_n^{\tilde{z}}))
        \]
        and consider a sequence $\{\gamma_n\}_{n\in\NN}\subseteq \Gamma$ such that $\lim_{n \to +\infty} \pi_1(\gamma_n)=e_1$, then $\lim_{n \to +\infty} \gamma_n(x)=x$ for all $x\in Y$. By Lemma \ref{lem_inverse_image_halfspace_positive_measure}, all of $A_{\tilde{x}}$, $A_{\tilde{y}}$ and $A_{\tilde{z}}$ are of positive measures in $B_1$. Hence, for any $\epsilon$, there exists $k\in \NN$ such that, for any $i \geq k$,
        \[
       \max ( \nu(A_{\tilde{x}}\Delta(\gamma_i.A_{\tilde{x}})),A_{\tilde{y}}\Delta(\gamma_i.A_{\tilde{y}})),A_{\tilde{z}}\Delta(\gamma_i.A_{\tilde{z}})))<\varepsilon,
        \]
        where $\nu$ is the measure on $B$. Hence, for any $i\in\NN$ large enough, the intersection 
       \[
       \hh_n^{\tilde{x}}\cap (\gamma_i^{-1}.\hh_n^{\tilde{x}}),\ \hh_n^{\tilde{y}}\cap (\gamma_i^{-1}.\hh_n^{\tilde{y}})\ \text{and}\ \hh_n^{\tilde{z}}\cap (\gamma_i^{-1}.\hh_n^{\tilde{z}})
       \]
        are not empty. Hence after fixing for each $i\in\NN$ large enough a triple of points $x_i,y_i,z_i$ such that $x_i\in \hh_n^{\tilde{x}}\cap (\gamma_i^{-1}.\hh_n^{{\tilde{x}}})$, $y_i\in \hh_n^{\tilde{y}}\cap (\gamma_i^{-1}.\hh_n^{\tilde{y}})$ and $z_i\in \hh_n^{\tilde{z}}\cap (\gamma_i^{-1}.\hh_n^{\tilde{z}})$, we obtain
         \[
        \lim_{i\rightarrow +\infty} \gamma_i(x_0)=\lim_{i\rightarrow +\infty} \gamma_i(m(x_i,y_i,z_i))=\lim_{i\rightarrow +\infty} m(\gamma_i(x_i),\gamma_i(y_i),\gamma_i(z_i)) =x_0
        \]
        which completes the proof.
        \end{proof}

        \begin{rmk}
           As in \cite{Shalom-superrigidity}, \cite{ChaIF} and \cite{Fior_superrigidity}, the invariant median space $Y$ can be defined as being the set of points $x\in X$ such that for any $\{\gamma_n\}_{n\in\NN}\subseteq \Gamma$ verifying $\displaystyle{\lim_{n \to +\infty} \pi_1(\gamma_n)=e_1}$, we have $\displaystyle{\lim_{n \to +\infty} \gamma_n(x)=x}$. One then uses the same argument as in the proof of Theorem \ref{thm_superrigidity_principal} to show that $Y$ is not empty, as it contains the median closure of the image of the boundary map $\phi$. 
        \end{rmk}

\section{Superrigidity of cocycles}\label{section_superrigidity_of_cocycles}

Let $G$ be a locally compact second countable group and let $(\Omega, \upsilon)$ be a standard probability space on which $G$ acts ergodically by measure-preserving transformations. Let $X$ be a complete separable irreducible median space of finite rank. Let $\alpha:G\times\Omega\rightarrow Isom(X)$ be a non elementary cocycle and $B$ the $G$-boundary. Note that the assumption on the cocycle $\alpha$ being non elementary imposes that $Isom(X)$ acts Roller non elementarily and Roller minimally on $X$. In particular, the regular boundary is not empty. \par 

We compose the cocycle $\alpha$ with the canonical isomorphism $Isom(X)\rightarrow Isom(Y)$ to consider the restriction of the action of $Isom(X)$ on the median core. We abuse notation and still denote the new cocycle by $\alpha$ and remark that it is also a non elementary cocycle. We refer to Section \ref{section_strongboundaries_cocycles} for the necessary background on cocycles and boundaries. The aim of this section is to deduce a superrigidity theorem for cocycles with values in the isometry group of a median space, namely, Theorem \ref{theoerem_superrigidity_cocycle}.\par 
Let $\HH'\subseteq \HH(X)$ be an $Isom(X)$-invariant fundamental family, that is, a dense family of halfspaces in the sense of Definition \ref{def_fundamental_family} whose existence is ensured by Proposition \ref{prop_existence_fundamental_family}.

\subsection{Construction of the boundary map}\label{subsection_boundary_map_cocycle}
The argument goes as in Section \ref{subsection_construction_boundary_map} to construct an $\alpha$-equivariant boundary map from $\Omega\times B$ to the regular boundary $\partial_r X$. As before, we denote by $\PP(\bar{Y})$ the set of probability measures on the closure of the median core $\bar{Y}\subseteq \bar X$. By the amenability of the action of $G$ on $\Omega\times B$ (see \cite[Proposition 4.3.4]{Zimmer}), there exists a $G$-equivariant map $\Phi:\Omega\times B\rightarrow \PP(\bar Y)$. Adapting the same ideas from Section \ref{subsection_construction_boundary_map}, we derive from it an $\alpha$-equivariant section from $\Omega\times B$ to $\bar Y$.
\begin{prop}\label{prop_boundary_map_cocycle}
There exists a $\alpha$-equivariant map from $\Omega\times B$ to $\bar Y$.
\end{prop}
Before proving the proposition, let us first prove the following lemmas.
\begin{lem}\label{lem_Fubini1}
    For any $\varepsilon,\eta>0$ and $\hh\in\HH(Y)$, there exists $\hh'\in\HH'$ such that $\hh^c\subseteq \hh'$ and
    \[
    (\Omega\times B)\backslash \Phi^{-1}(ev_{\hh'}^{-1}(]1-\varepsilon,1]))
    \]
    is of a measure bigger than $\eta$.
\end{lem}
\begin{proof}
    The proof follows the same line as the proof of Lemma \ref{lem_measures_having_consistant_support_in_halfspace}.
\end{proof}

\begin{lem}\label{lem_Fubini2}
    For any fixed thick halfspace $\hh\in \HH(Y)$ and a fixed $\epsilon>0$. Then for almost all $\omega \in\Omega$ the following 
    \[
    B_{\hh,\omega}^\varepsilon:=\{b\in B\ | \ ev_\hh(\Phi(\omega,b))>1-\varepsilon\}
    \]
    is of positive measure.
\end{lem}
\begin{proof}
    For the sake of contradiction, we assume that there exists a subset $A\subseteq \Omega$ of positive measure such that for any $\omega \in A$, the subset 
    \[
     B_{\hh,\omega}^\varepsilon:=\{b\in B\ | \ ev_\hh(\Phi(\omega,b))>1-\varepsilon\}
    \]
    is of measure zero. By Lemma \ref{lem_Fubini1}, there exists $\hh'\in\HH(Y)$ and a subset $A'\subseteq A$ of positive measure such that for all $\omega'\in A$ the subset $B_{\hh',\omega'}^\varepsilon$ is of positive measure. Note that by Remark \ref{rmk_measures_having_consistant_support_in_halfspace}, $\hh'$ can be chosen such that $\hh^c\subseteq \hh'$. The cocycle $\alpha$ being non elementary, there exists $g\in G$ and $\omega'\in A'\cap g^{-1}.A'$ such that $\alpha(g,\omega').\hh'\subseteq \hh'^c$. This implies that $g.B_{\hh',\omega'}\subseteq B_{\hh,g.\omega'}$, contradicting the assumption that $B_{\hh,\omega}$ is of zero measure for any $\omega\in A$, which completes the proof of the lemma.  
\end{proof}
\begin{proof}[Proof of Proposition \ref{prop_boundary_map_cocycle}]
    As before, we denote by $\PP_1(\bar{Y})$ the set of probability measures whose center of mass is a single point. We first show that the image of $\Phi$ lies essentially in $\PP_1(\bar Y)$.  We consider the following function 
    \begin{eqnarray*}
            &&\psi: \Omega\times B\times B \rightarrow \RR,\\
            &&(\omega,b_1,b_2) \mapsto \inf_{h\in\HH'}{\{|ev_\hh(\Phi_{\omega}(b_1))-ev_{\hh^c}(\Phi_{\omega}(b_1))|+|ev_\hh(\Phi_{\omega}(b_2))-ev_{\hh^c}(\Phi_{\omega}(b_2))|\}}
        \end{eqnarray*}
  where we simplify $\Phi(\omega,b)$ by $\Phi_{\omega}(b)$. As in the proof of Proposition \ref{prop_balanced_halfspaces_zero_measure}, we first show that the value of $\Psi$ is bigger than $1$ almost everywhere. For two strongly separated thick halfspaces $\hh_1,\hh_2$, Lemma \ref{lem_Fubini2} ensures the existence of a subset of positive measure $K\subseteq \Omega\times B\times B$ such that for any $(\omega,b_1,b_2)\in K$, we have $\Phi(\omega,b_1)\in ev_{\hh_1}^{-1}(]1-\varepsilon,1[$ and $\Phi(\omega,b_2)\in ev_{\hh_2}^{-1}(]1-\varepsilon,1[$. Hence, the restriction of $\Psi$ over $K$ has a value larger than $1-2\varepsilon$. By the ergodicity of the action of $G$ on $\Omega\times B\times B$ and the $G$-invariance of the map $\psi$, we deduce that it takes essentially a value greater than $1-2\varepsilon$. The same argument applies for any $\varepsilon>0$, we deduce that the value of $\Psi$ is essentially greater or equal $1$.\par 
  By Lemma \ref{lem_countable_balanced_halfspaces}, to show that the image of $\Phi$ is essentially in $\PP_1(\bar Y)$, it is enough to show that $\Phi^{-1}(ev_\hh^{-1}(\{\frac{1}{2}\})) $ is of zero measure for any $\hh\in\HH'$. For the sake of contradiction, let us assume that there exists $\hh\in\HH'$ such that $\Phi^{-1}(ev_\hh^{-1}(\{\frac{1}{2}\}))$ has a positive measure. Now we set $A:=\Phi^{-1}(ev^{-1}(\{\frac{1}{2}\}))$ and remark that $\pi_{\Omega}(A)$ is of positive measure (as $A$ is). By Fubini's Theorem, for almost all $w\in \pi_\Omega(A)$ the subset 
  \[
  B_\omega^A:=\{b\in B\ |\ (\omega,b)\in A\}
  \]
  is of positive measure. Hence, again by Fubini's theorem, the subset 
  \[
  K:=\{(\omega,(b_1,b_2))\in \Omega\times ( B\times B)\ |\ \pi_\Omega(A)\ \text{and} \ (b_1,b_2)\in B_\omega^A\times B_\omega^A  \}
  \]
  is of positive measure.\par

  However, the restriction of $\Psi$ in $K$ is equal to zero, which is a contradiction and completes the proof of the proposition.\par 
  Therefore, the map $\Phi$ descends to a $\alpha$-equivariant map $\phi: \Omega\times B\rightarrow \bar Y$. By the metric ergodicity of the $G$-action on $\Omega\times B$ and the measurability of the decomposition $\bar Y=Y\sqcup \partial Y$, we deduce that the image of $\phi$ lies essentially in $\partial Y$.
\end{proof}
\begin{prop}
    The image of the map $\phi$ lies essentially in $\partial_r Y$.
\end{prop}
    \begin{proof}
        The proof follows the same idea as the proof of Theorem \ref{thm_image_in_regular_boundary} with the technical adjustment needed for the language of cocycles. We consider the following map 
        \begin{eqnarray*}
            f: \Omega\times B\times B &\rightarrow& \partial Y^2\\
            \ \ \ (\omega,b_1,b_2)&\mapsto& (\phi(\omega,b_1),\phi(\omega,b_2))
        \end{eqnarray*}
        
        We fix $r>0$ and consider set of thick halfspaces $\{\hh_0,..,\hh_n \} \subseteq \HH'$ where $\hh_i^c$ and $\hh_{i+1}$ are strongly separated with $d(\hh_i^c,\hh_{i+1})>r$ for any $i\in\NN$. By Lemma \ref{lem_Fubini2}, there exists a subset $K\subseteq \Omega\times B\times B$ of positive measure such that for any $(\omega,b_1,b_2)\in K$, we have $\phi(\omega,b_1)\in\tilde{\hh}_0^c$ and $\phi(\omega,b_2)\in\tilde{\hh}_n$. By the ergodicity of the $G$-action on $\Omega\times B\times B$, we deduce that almost all $(\tilde{x}_1,\tilde{x}_2)\in f(\Omega\times B\times B)$ are separated by $n$ strongly separated $r$-spaced halfspaces. This is verified for any $n$, hence Lemma \ref{lem_characterization_regular_points} implies that for almost all $(\omega,b_1,b_2)\in \Omega\times B\times B$, we have either $\phi(\omega,b_1)\in\partial_r Y$ or $\phi(\omega,b_2)\in\partial_r Y$. We conclude then that the image of $\phi$ is essentially in $\partial_r Y$.
    \end{proof}

    \begin{prop}
        Any $\alpha$-equivariant map $\psi:\Omega\times B\times B\rightarrow\partial_r Y$ is of the form $\psi(\omega,b_1,b_2)=\phi(\omega,b_1)$ for almost all $(\omega,b_1,b_2)\in \Omega\times B\times B$ or $\psi(\omega,b_2)$ for almost all $(\omega,b_1,b_2)\in \Omega\times B\times B$.
    \end{prop}
    \begin{proof}
        Let us consider an $\alpha$-equivariant map  $\psi:\Omega\times B\times B\rightarrow\partial_r Y$   and set 
        \begin{eqnarray*}
            \Psi: \Omega\times B\times B &\rightarrow& \bar Y\\
            \ \ \ (\omega,b_1,b_2)&\mapsto& m(\phi(\omega,b_1),\phi(\omega,b_2),\psi(\omega,b_1,b_2))
        \end{eqnarray*}
        The result follows directly from Lemma \ref{lem_interval_between_strongly_separated_ultrafilters}, Proposition \ref{prop_no_equivariant_section} and  Lemma \ref{lem_metric_ergodicity_cocycle}.
    \end{proof}
Hence, one gets the cocycle analogue of Corollary \ref{cor_boundary_map_into_product}.
    \begin{cor}\label{cor_boundary_map_into_product_cocycle}
Any $\alpha$-equivariant measurable map from $\psi: \Omega\times B\times B\rightarrow (\partial_r Y)^2$ is of the following form $\upsilon$-almost everywhere: 
\begin{enumerate}
 \item $\psi(\omega,b_1,b_2)=(\phi(\omega,b_1),\phi(\omega,b_2))$ (or $\psi(\omega,b_1,b_2)=(\phi(\omega,b_2),\phi(\omega,b_1))$).
 \item $\psi(\omega,b_1,b_2)=(\phi(\omega,b_i),\phi(\omega,b_i))$ where $i\in\{1,2\}$ is fixed.
\end{enumerate}
\end{cor}

\subsection{Superrigidity}
In this section, we restrict ourself to the case where $G$ is the product of two locally compact second countable groups that are not compact.
\begin{prop}\label{prop_boundary_map_factors}
    Let $G:=G_1\times G_2$ where $G_1$ and $G_2$ are locally compact, second countable, non compact groups. Then the boundary maps $\phi$ factors through the projection onto $\Omega\times B_i$ for one of the $B_i$'s.
\end{prop}
\begin{proof}
    Let us consider the following map 
    \begin{eqnarray*}
    \eta:\Omega\times B_1\times B_2 \times B_1 \times B_2 &\rightarrow & (\partial_r Y)^2\\
\ \ \ \ (\omega,x,y,x',y') &\mapsto & (\phi(\omega,x,y),\phi(\omega,x,y'))
    \end{eqnarray*}
    Corollary \ref{cor_boundary_map_into_product_cocycle} states that we have three cases:
    \begin{enumerate}
        \item $\phi(\omega,x,y')=\phi(\omega,x',y')$ almost everywhere, implying that $\phi$ does not depends on the $y$ variable.
        \item $\phi(\omega,x,y')=\phi(\omega,x,y)$ almost everywhere, which implies that $\phi$ does not depend on the $x$ variable.
    \end{enumerate}
    The case where $\phi(\omega,x,y)=\phi(\omega,x',y')$ cannot happen under the assumption that the cocycle is not elementary.
\end{proof}

\begin{thm}\label{theoerem_superrigidity_cocycle}
    There exists a group homomorphism $h:G\rightarrow Isom(Y)$ that factors through one of the $G_i$'s, and $h$ is cohomologous to the cocycle $\alpha$
\end{thm}
\begin{proof}
Let us denote by $Y_0\subseteq Y$ the $\alpha$-invariant subset obtained by taking the median points of distinct triple in $\partial_r X\subseteq \bar Y$ and let us denote by $Y'$ its median hull inside $X$ as in Definition \ref{def_median_core}. The median subspace $Y'$ is countable and $\alpha$-invariant (see Theorem \ref{thm_countable_medians_subspace}). \par 
 By Proposition \ref{prop_boundary_map_factors}, the boundary map $\phi$ factors through $\Omega\times B_i$ for one of the $B_i$'s. Let us assume that it is $\Omega\times B_1$. Then almost all fixed $b_1$ gives rise to an $\alpha_{G_2}$ equivariant section, where $\alpha_{G_2} = \alpha|_{G_2}$ is the cocycle on $G_2$ obtained as the restriction of $\alpha$ to $\{e_{G_1}\}\times G_2$. We use the latter familly to construct infinitely many $\alpha_{G_2}$-equivariant section from. Let $\YY$ be the set of equivalence classes of measurable $\alpha_{G_2}$-invariant sections from $\Omega$ to $Y'$, where two measurable section are equivalent if they coincide in a subset of total measure.\par 
 \textit{\textbf{Claim 1:} The set $\YY$ is not empty. }\par 
 Let us fix $y\in Y_0$ and consider $\tilde{x}_1,\tilde{x}_2,\tilde{x}_3\in\partial_r X$ such that $y=m(\tilde{x}_1,\tilde{x}_2,\tilde{x}_3)$. There exists then three strongly separated thick halfspaces $\hh_1,\hh_2,\hh_3\in\HH(X)$ such that $\tilde{x}_i\in\tilde{\hh}_i$ and for any $\tilde{x}_i'\in\tilde{\hh}_i$ we have $m(\tilde{x}_1,\tilde{x}_2,\tilde{x}_3)=y$. Then for any measurable subset $A\subseteq \Omega$ of positive measure, Lemma \ref{lem_Fubini2} and Fubini's Theorem ensures the existence of a measurable subset $A'\subseteq A$ of positive measure  and $b_1,b_2,b_3\in B_1$ such that for any $\omega\in A'$, we have $\phi(\omega,b_i)\in\tilde{\hh}_i$. The measurable decomposition $m(\partial_r X):=\partial_r X\sqcup Y_0$ is $Isom(X)$-invariant. Hence, by the $\alpha_{G_2}$-equivariance of the sections and the ergodicity of the $G_2$ action on $\Omega$, we deduce that for almost all $b_1,b_2,b_3\in B$ we have $m(\phi(\omega,b_1),\phi(\omega,b_2),\phi(\omega,b_3))\in Y_0$, which gives a well defined class of $\alpha_{G_2}$-invariant measurable section from $\Omega$ to $Y_0$ which completes the first claim.  
 \par
 The action of $G_2$ on $\Omega$ is ergodic and the target space $Y'$ is countable, the set $\YY$ is also countable. Hence let us fix a choice of representative of each class of $\YY$ to obtains a sequence $(f_i)_{i\in\NN}$, that is $\YY:=\displaystyle{\bigcup_{i\in\NN}[f_i]}$, where $[f_i]$ is the equivalent class of $f_i$.\par 
  \textit{\textbf{Claim 2:} Almost everywhere, the evaluation map give rise to a surjection between $\YY$ and $Y'$.}\par
  In the proof of the first claim, we have shown that for any measurable subset $A\subseteq \Omega$ of positive measure and any $y\in Y'$, there exists a measurable subset of positive measure $A'\subseteq A$ and $f\in \YY$ such that $f(\omega)=y$ for any $\omega\in A'$ (note that we did it when $y$ lies in $Y_0$ but the same argument applies for any point in $Y'$ by iterating the median operator on the $\alpha_{G_2}$ invariant sections obtained in the previous step and taking value in $Y_0$). Hence, the measurable subset 
  \[
  A_y:=\bigcup_{i\in\NN}f_i^{-1}(y)
  \]
 is of total measure, as it is measurable and it intersects any measurable subset of positive measure in $\Omega$.\par 
 After setting $\Omega':=\displaystyle{\bigcap_{y\in Y'}A_y}$, we deduce that for each $\omega\in \Omega'$ the evaluation map $ev_\omega:\YY\rightarrow Y'$ is surjective.\par 
 By the ergodicity of the $G_2$-action on $\Omega$ and the equivariant of the sections $f_i's$, $\YY$ is naturally endowed with a median metric $d_\YY(f_i,f_j):=\int_{\Omega}d_Y(f_i(\omega),f_j(\omega))d\mu(\omega)$. \par 
 
\textit{\textbf{Claim 3:} Almost everywhere, the evaluation map give rise to an isometric identification between $\YY$ and $Y'$.}\par 
For any $i,j\in\NN$, the function $d(f_i(*),f_j(*))$ is $G_2$-invariant, hence it is constant by the metric ergodicity of the $G_2$-action on $\Omega$. Thus, there exists a subset $\Omega''\subseteq \Omega$ of total measure, such that $f_i(\omega)\neq f_j(\omega)$ for any $\omega\in\Omega''$ and any distinct $i,j\in\NN$ (the family $(f_i)_{i\in\NN}$ is assumed to be pairwise essentially distinct). Therefore, the evaluation map $ev_\omega:(\YY,d_\YY)\rightarrow (Y',d_Y)$ is an isometric bijection for every $\omega\in \Omega''$.\par 

After fixing a base point $\omega_0$ from which we identify $\YY$ with $Y'$, we have a measurable map 
\begin{eqnarray*}
    F:\Omega'' \times Y' &\rightarrow & Y'\\
\ \ \ \ (\omega,f_i(\omega_0)) &\mapsto & f_i(\omega)
    \end{eqnarray*}
Hence, each fixed $\omega\in\Omega''$ gives rise to an element in $Isom(Y')$. Thus, we have a measurable map $h:\Omega\rightarrow Isom(Y')\cong Isom(Y)$.\par 
Therefore, the $G_1$-isometric action on $\YY$ is cohomologous to the cocycle $\alpha$ under conjugation by $h$.
\end{proof}

\begin{rmk}
    Another way to go for the proof of Theorem \ref{theoerem_superrigidity_cocycle} is to first fix, as in the Proof of the theorem, a parametrization  $\YY:=\displaystyle{\bigcup_{i\in\NN}[f_i]}$ of the countable space of $\alpha_{G_2}$-equivariant sections from $\Omega$ to $Y_0$. Then one applies the reduction lemma to transform each $f_i$ into an essentially constant section. As $Y_0$ is countable, orbits are canonically separated by countably many measurable subsets of $Y_0$ (the measurable subsets being the orbits themselves). Hence, if $f$ is an  $\alpha_{G_2} $-equivariant section, by the ergodicity of $G_2$ action on $\Omega$, the image of the section lies essentially in one orbit. Thus, one obtains a $G_2$ equivariant map $\tilde{f}:\Omega\rightarrow Isom(X)/stab(y_f)$.  Hence, up to conjugating the cocycle by $h_f:\Omega\rightarrow Isom(X)$ that associates to each $ \omega$ a representative of the coset $\tilde{f}(\omega)$ (as in \cite[Example 4.2.18 (b)]{Zimmer}), the new cocycle has a constant equivariant section $\hat{f}(\omega)=y_f$ for almost all $\omega\in \Omega$. Hence, the new cocycle verifies 
    \[
    \tilde{\alpha}(g_2,\omega):=h_f^{-1}(g_2.\omega)\alpha(g_2,\omega)h(\omega)\in Stab(y_f)\ \ \text{for almost all}\ g_2\in G_2,\omega\in \Omega
    \]
   Apply this process for each $f_i$, where at each step, one conjugates the cocycle by a map $h_i:\Omega\rightarrow Stab(x_0\cup x_1\cup ..\cup x_i)$ such that $[f_i]$ is mapped into a constant section. One set $\hat{h}_i:=h_i\circ ..\circ 
    h_1\circ h_0$. As $\displaystyle{\bigcap_{i\in\NN} stab(x_0\cup x_1\cup ..\cup x_i)=\{\id_{Y_0}\}}$. Hence, the sequence $(\hat{h}_i)_{i\in\NN}$ is a Cauchy sequence. The space $Isom(X)$ being separable and completely metrizable, we deduce that the sequence converges to an element $h:\Omega\rightarrow Isom(Y_0)$ where $(h(g_2.\omega))^{-1}\alpha(g_2,\omega)h(\omega)=Id_Y$ for almost all $g_2\in G_2, \omega\in \Omega$. 
\end{rmk}

We now prove Theorem \ref{thm_elementary_cocycyle_product_group}. The proof follows basically the same line as in \cite[Theorem C]{Fior_superrigidity}.
\begin{proof}[Proof of Theorem \ref{thm_elementary_cocycyle_product_group}]
 By Theorem \ref{theoerem_superrigidity_cocycle}, the cocycle $\alpha$ is cohomologous to a continuous group homomorphism $h:G\rightarrow Isom(Y)$ that factors through one of the $G_i$'s, say $G_1$. By Proposition \ref{prop_elementarity_cocycle_vs_homomorphism}, it is enough to prove that the $G_1$-action induced from $h$ is Roller elementary. Let $Y'$ be as in Definition \ref{def_median_core}. By the third point of Remark \ref{remark_median_core}, the stabilizer of any point in $Y'$ is open and closed. Hence for any $y\in Y'$, we have $h(G_1^0)\subseteq Stab(y)$. Therefore, the $G_1^0$-action on $Y'$ is trivial which extends to a trivial action on the median core $Y$. The $G_1$ action factors through the quotient $G_1/G_1^0$. In the case where $G_1/G_1^0$ has Kazhdan's Property (T), the Roller elementarity of the action is deduced from \cite[Theorem 1.2]{ChaDH_median}  and \cite[Corollary 2.17]{Fior_tits}. If $G_1/G_1^0$ is amenable, then the Roller elementarity of the action is implied by \cite[Theorem E]{Fior_tits}, which completes the proof.
\end{proof}

\section{Cocycle superrigidity for Kazhdan groups}\label{section_Kazhdan_group}
In previous sections, we consider cocycles whose source groups are product groups; they include, for instance, products of simple Lie groups. However, they exclude simple Lie groups of higher rank. In this section, we complete the story by studying cocycles whose source groups have Kazhdan's Property (T). The reader is referred to \cite{BHV} for examples and other aspects of groups having Kazhdan's Property (T).

The following theorem is essentially included in the proof of \cite[Theorem 2.3]{AdamsSpatzier}. Recall that for a set $X$, a function $f: X \times X \to \mathbb{R}$ is called \textbf{negative definite} if for all $x_1,...,x_n$ in $X$ and all complex numbers $\lambda_1,...,\lambda_n$ with $\sum^n_{i=1}\lambda_i = 0$, $$\sum^n_{i,j =1 } \lambda_i\overline{\lambda_j}f(x_i,x_j) \leq 0.$$
\begin{thm}[\cite{AdamsSpatzier}]\label{thm_Adams-Spatzier}
    Let $G$ be a lcsc group having Kazhdan's Property (T) and $(\Omega,\nu)$ an ergodic $G$-space. Let $(X,d)$ be a metric space such that the distance function $d: X \times X \to \mathbb{R}$ is negative definite.  Let $\alpha: G \times \Omega \to Isom(X)$ be a measurable cocycle. Then there exists an $\alpha$-invariant field of bounded subsets.
\end{thm}\par 
We note that in \cite{AdamsSpatzier}, Theorem \ref{thm_Adams-Spatzier} is stated for discrete lattices but as pointed out by the authors therein, the same result holds for locally compact second countable groups.\par 
Here, an \textbf\textit{{$\alpha$-invariant field of (convex) bounded subsets}} is a measurable subset $M \subset X \times \Omega$ such that if one sets $M_{\omega} = \{ x \in X| (\omega, x) \in M\}$, then we have 
\begin{itemize}
    \item [(1)]$\alpha(g,\omega)F_{\omega} = F_{g.\omega}$ holds for all $g \in G$ and almost all $\omega \in \Omega$, and
    \item [(2)] For almost all $\omega \in \Omega$, $F_{\omega} \subset X$ is a (convex) bounded subset. 
\end{itemize}
The following lemma is included in \cite{nica2008groupactionsmedianspaces} and \cite{ChaDH_median}, for CAT(0) cube complexes, it is included in \cite{Niblo-Reeves}.
\begin{lem}
    If $(X,d)$ is a metric median space, then the distance function $d: X \times X \to \mathbb{R}$ is negative definite.
\end{lem}
Combining with the previous theorem, we have
\begin{thm}\label{thm:cocycle-propertyT}
    Let $G$ be a lcsc group having Kazhdan's Property (T) and $(\Omega,\nu)$ an ergodic $G$-space. Let $(X,d)$ be a median space and $\alpha: G \times \Omega \to Isom(X)$ a measurable cocycle. Then there exists an $\alpha$-invariant field of bounded subsets. 
\end{thm}

This theorem, together with the same argument for \cite[Theorem 1.1]{AdamsSpatzier} and the fact that any finite subcomplex can be completed to be a sub-cube complex, shows the following.

\begin{cor}\label{cor:cocycle_cube_propertyT}
   Let $G$ be a lcsc group having Kazhdan's Property (T) and $(\Omega,\nu)$ an ergodic $G$-space. Let $X$ be a CAT(0) cube complex and $H \leq Isom(X)$. Then any measurable cocycle $\alpha: G \times \Omega \to H$ is cohomologous to a cocycle in the isotropy group in $H$ of a bounded sub-cube complex of $X$.
\end{cor}

In the case where $X$ is of finite rank, we prove then the following fixed point property in the cocycle setting under the assumption that $X$ is connected.
\begin{thm}\label{thm_equivariant_cocycle_Kazhdan}
     Let $X$ be a complete separable connected median space of finite rank and let $G$ be a lcsc group with Kazhdan's Property (T) that acts ergodically by measure preserving action on $\Omega$. Then for any cocycle $\alpha:G\times\Omega\rightarrow Isom(X)$ there exists an $\alpha$-equivariant section $f:\Omega\rightarrow X$. \par 
     In particular, $\alpha$ is elementary.
\end{thm}
The idea of the proof is to start from the $\alpha$-invariant measurable field of bounded subsets of $X$ given by Theorem \ref{thm:cocycle-propertyT}. Then one associates to each fiber, which is a bounded subset, its centroid in the same vein as the barycenter map given in Section \ref{section_barycenter_map}, to obtain an $\alpha$-equivariant section. What is left then is to show its measurability.\par
Before proving Theorem \ref{thm_equivariant_cocycle_Kazhdan}, we first describe the process to obtain the centroid. \par
Let $X$ be a complete finite rank median space and let $K\subseteq X$ be a bounded subset. For any halfspace $\hh\in\HH(X)$ define the \textbf{\textit{depth}} of $\hh$ in $K$ to be 
\[
depth_K(\hh):=\sup_{x\in K}{(d(x,\hh^c))}.
\]
For any closed convex subset $K$, consider the following subset of halfspaces 
\[
\HH^+_K:=\{\hh\in\HH(X)\ | \ depth_K(\hh)\geq depth_K(\hh^c)\}.
\]
Remark that as $X$ is connected, if $\hh_1,\hh_2\in\HH_K^+$ are disjoint, then we necessarily have $\bar\hh_1\cap\bar\hh_2\neq\emptyset$, where $\bar{\hh}$ is the closure of $\hh$ in $X$.\par
We define the \textbf{\textit{centroid}} of $C$ to be 
\[
c(K):=\bigcap_{\hh\in\HH_K^+}\bar \hh.
\]
Remark that any $\hh\in \HH_K^+$ intersects the closure of the convex hull of $K$, which is bounded and closed. Hence, by the compactness of $\bar X$ and Helly's property, we deduce that $c(K)$ is not empty. Moreover it is a singleton as $X$ is connected, any pair of points is separated by a halfspace that has different depth than its complement in $K$. Note that for every $f\in Isom(X)$, we have $c(f.K)=f(c(K))$.\par 
\begin{proof}[Proof of Theorem \ref{thm_equivariant_cocycle_Kazhdan}]
    Let $M\subseteq \Omega\times X$ be a measurable field of bounded subsets obtained from Theorem \ref{thm:cocycle-propertyT}. Let $\Omega_0\subseteq \Omega$ be the subset of total measure such that, for any $\omega\in\Omega_0$, the fiber $M_\omega:=\{x\in X\ |\ (\omega,x)\in M\}  $ is a bounded subset. Note that up to taking an $\varepsilon$ neighborhood of each fiber $M_\omega$, we can assume that its interior $int(M_\omega)$ is dense in it. \par 
    By considering the centroid of each fiber $M_\omega$ for $\omega\in \Omega_0$, one obtains an $\alpha$-equivariant section $f:\Omega\rightarrow X$. It is left to verify that it is measurable.\par 
    To show that $f$ is measurable, it is enough to show that its graph is a measurable subset of $\Omega\times X$ (see \cite[Theorem A5]{Zimmer}). We will show that the graph of $f$ is obtained as a countable intersection of measurable subset of $\Omega\times X$. To do so, we apply the same process in the definition of centroid, but with a fundamental family of halfspaces and we do it uniformly with regard to the index space $\Omega$, to obtain at each step, a measurable subset of $M$.\par 
    As $X$ is separable, let us fix a countable dense subset $Y\subseteq X$ and  a fundamental family $\HH'\subseteq \HH(X)$. Set $A:=M\cap(\Omega\times Y)$ and consider a countable family of measurable section $\phi_n:\Omega\rightarrow X$ such that such that the union of their graphs is equal to $A$, up to a null set. The existence of such family is ensured by the Lusin-Novikov Theorem.\par 
    For each $\hh\in\HH'$, consider the following measurable function
    \begin{eqnarray*}
        d_\hh:\Omega&\rightarrow&\RR\\ 
        \omega &\mapsto& \sup_{n\in\NN}(d(\phi_n(\omega),\hh^c))
    \end{eqnarray*}
    and set $\Omega_\hh:=(d_\hh-d_{\hh^c})^{-1}([0,+\infty[)$. Remark that for any $\omega\in\Omega_\hh$, the centroid $c(M_\omega)$ of the fiber $M_\omega$ lies in $\hh$, as we have 
   \[
   \sup_{x\in M_\omega}(d(x,\hh^c))=\sup_{n\in\NN}(d(\phi_n(\omega),\hh^c))\geq \sup_{n\in\NN}(d(\phi_n(\omega),\hh))=\sup_{x\in M_\omega}(d(x,\hh))
   \]
    implying that $\hh\in \HH_{M_\omega}^+$.
    Hence, the graph of the centroid map $f$ lies essentially in $(\Omega_\hh\times\hh)\cup(\Omega_{\hh^c}\times\hh^c)$. \par 
    Therefore, we conclude that up to zero measure set of $\Omega\times X$, we have
    \[
    f(\Omega)\subseteq\bigcap_{\hh\in\HH'} (\Omega_\hh\times\hh)\cup(\Omega_{\hh^c}\times\hh^c).
    \]
  
    By the density of the fundamental family, we have 
   \[
c(K):=\bigcap_{\hh\in\HH_K^+}\bar \hh=\bigcap_{\hh\in\HH_K^+\cap \HH'}\bar \hh.
\]
Therefore, we get 
\[
    f(\Omega)=\bigcap_{\hh\in\HH'} (\Omega_\hh\times\hh)\cup(\Omega_{\hh^c}\times\hh^c)
    \]
    which completes the proof.
    \end{proof}

\newpage
\bibliographystyle{plain}
\bibliography{bibli}

\bigskip
\noindent


\end{document}